\pdfoutput=1
\documentclass[english,letterpaper,reqno,oneside]{amsart}

\expandafter\let\csname ver@amsthm.sty\endcsname\relax

\usepackage[colorlinks,citecolor=blue,urlcolor=blue,linkcolor=RawSienna,hypertexnames=true]{hyperref}
\usepackage{amsthm,mathtools,thmtools}
\usepackage[capitalize]{cleveref}

\usepackage[T1]{fontenc}
\usepackage[utf8]{inputenc}
\usepackage{setspace}
\usepackage{textcomp}
\usepackage{enumitem}
\usepackage{amsmath, amssymb,bm, cases, mathtools, thmtools}
\usepackage{graphicx}
\usepackage{esint} %
\usepackage{xargs}[2008/03/08]
\usepackage{datetime}
\usepackage[normalem]{ulem}
\usepackage{mathrsfs}
\usepackage[font=small]{caption}
\usepackage[usenames,dvipsnames]{xcolor}

\usepackage{euscript,microtype}

\usepackage{csquotes}
\usepackage[%
    minnames=4,maxnames=99,maxcitenames=4,
    style=alphabetic,
    doi=false,url=false,
    firstinits=true,hyperref,natbib,backend=bibtex]{biblatex}
\renewbibmacro{in:}{%
  \ifentrytype{article}{}{\printtext{\bibstring{in}\intitlepunct}}}
\AtEveryBibitem{\clearfield{issn}}
\AtEveryBibitem{\clearfield{isbn}}
\AtEveryCitekey{\clearfield{issn}}
\AtEveryCitekey{\clearfield{isbn}}
\AtEveryBibitem{\ifentrytype{book}{\clearfield{pages}}{}}
\DeclareFieldFormat[book,article,unpublished,misc]{title}{\textit{#1}}
\DeclareFieldFormat[article]{journaltitle}{{#1}}

\bibliography{random_graphs}

\crefname{lem}{Lemma}{Lemmas}
\crefname{cor}{Corollary}{Corollaries}
\crefname{thm}{Theorem}{Theorems}
\crefname{assumption}{Assumption}{Assumptions}

\numberwithin{equation}{section}

\declaretheorem[style=plain,numberwithin=section,name=Theorem]{thm}
\declaretheorem[style=plain,sibling=thm,name=Lemma]{lem}
\declaretheorem[style=plain,sibling=thm,name=Proposition]{prop}
\declaretheorem[style=plain,sibling=thm,name=Corollary]{cor}

\makeatletter
\def\@endtheorem{\endtrivlist}%
\makeatother

\declaretheorem[style=definition,sibling=thm,name=Definition]{defn}

\declaretheorem[style=remark,qed=$\triangleleft$,sibling=thm,name=Remark]{remark}

\makeatletter
\def\@endtheorem{\endtrivlist\@endpefalse }%
\makeatother

\numberwithin{thm}{section}

\newlist{casenv}{enumerate}{4}
 \setlist[casenv]{leftmargin=*,align=left,widest={iiii}}
 \setlist[casenv,1]{label={{\itshape\ \casename} \arabic*.},ref=\arabic*}
 \setlist[casenv,2]{label={{\itshape\ \casename} \roman*.},ref=\roman*}
 \setlist[casenv,3]{label={{\itshape\ \casename\ \alph*.}},ref=\alph*}
 \setlist[casenv,4]{label={{\itshape\ \casename} \arabic*.},ref=\arabic*}

\def\[#1\]{\begin{align}#1\end{align}}
\def\*[#1\]{\begin{align*}#1\end{align*}}

\newcommand{\LATER}[1]{\error}
\newcommand{\fLATER}[1]{\error}
\newcommand{\TBD}[1]{\error}
\newcommand{\fTBD}[1]{}
\newcommand{\PROBLEM}[1]{\error}
\newcommand{\fPROBLEM}[1]{\error}
\newcommand{\NA}[1]{#1}

\usepackage{babel}

 \providecommand{\casename}{Case}

\begin{document}

\newcommand{\ErdosRenyi}{Erd\H{o}s--R\'enyi--Gilbert}

\newcommand{\defnphrase}[1]{\emph{#1}}

\global\long\def\defas{\vcentcolon=}

\global\long\def\st{\,:\,}

\global\long\def\dist{\ \sim\ }

\global\long\def\given{\mid}

\global\long\def\distiid{\overset{iid}{\dist}}

\global\long\def\distind{\overset{ind}{\dist}}

\global\long\def\Naturals{\mathbb{N}}

\global\long\def\Rationals{\mathbb{Q}}

\global\long\def\Reals{\mathbb{R}}

\global\long\def\BorelSets{\mathcal{B}}

\global\long\def\Nats{\mathbb{N}}

\global\long\def\Ints{\mathbb{Z}}

\global\long\def\NNInts{\Ints_{+}}

\global\long\def\NNExtInts{\overline{\Ints}_{+}}

\global\long\def\Cantor{\2^{\mathbb{N}}}

\global\long\def\NNReals{\Reals_{+}}

\global\long\def\as{\textrm{ a.s.}}

\global\long\def\epi{\textrm{epi}}

\global\long\def\intr{\textrm{int}}

\global\long\def\conv{\textrm{conv}}

\global\long\def\cone{\textrm{cone}}

\global\long\def\aff{\textrm{aff}}

\global\long\def\cone{\textrm{cone}}

\global\long\def\dom{\textrm{dom}}

\global\long\def\cl{\textrm{cl}}

\global\long\def\ri{\textrm{ri}}

\global\long\def\grad{\nabla}

\global\long\def\imp{\Rightarrow}

\global\long\def\downto{\!\downarrow\!}

\global\long\def\upto{\!\uparrow\!}  \global\long\def\AND{\wedge}

\global\long\def\OR{\vee}

\global\long\def\NOT{\neg}

\global\long\def\PowerSet{\mathcal{P}}

\global\long\def\Measures{\mathcal{M}}

\global\long\def\ProbMeasures{\mathcal{M}_{1}}

\global\long\def\equaldist{\overset{d}{=}}

\global\long\def\equalprob{\overset{p}{=}}

\global\long\def\inv{^{-1}}

\global\long\def\norm#1{\lVert#1 \rVert}

\global\long\def\event#1{\left\lbrace #1 \right\rbrace }

\global\long\def\tuple#1{\langle#1 \rangle}

\global\long\def\bspace{\Omega}

\global\long\def\bsa{\mathcal{A}}

\global\long\def\borelspace{(\bspace,\bsa)}

\global\long\def\card#1{\##1}

\global\long\def\iid{i.i.d.\ }

\global\long\def\iff{iff\ }

\global\long\def\gprocess#1#2{(#1)_{#2}}

\global\long\def\nprocess#1#2#3{\gprocess{#1_{#3}}{#3 \in#2}}

\global\long\def\process#1#2{\nprocess{#1}{#2}n}

\newcommand{\EE}{\mathbb E}

\global\long\def\expect#1{\EE [#1]}

\global\long\def\var#1{\mbox{var}\left[#1\right]}

\global\long\def\equalas{\overset{\mathrm{a.s.}}{=}}

\global\long\def\abs#1{\lvert#1 \rvert}

\global\long\def\norm#1{\lVert#1\rVert}

\global\long\def\inedge#1{e_{#1}^{\mbox{in}}}

\global\long\def\outedge#1{e_{#1}^{\mbox{out}}}

\global\long\def\intd{\mathrm{d}}

\global\long\def\suchthat{\mid}

\global\long\def\exclude{\backslash}

\global\long\def\Pr{\mathrm{P}}

\global\long\def\convPr{\xrightarrow{\,p\,}}
\global\long\def\convD{\xrightarrow{\,d\,}}

\global\long\def\convDist{\xrightarrow{\,d\,}}

\global\long\def\floor#1{\lfloor#1\rfloor}

\global\long\def\asympLim#1{,\ #1\rightarrow\infty}

\global\long\def\diri{\mathrm{Diri}}

\global\long\def\categ{\mathrm{Cat}}

\global\long\def\betaDist{\mathrm{Beta}}

\global\long\def\bern{\mathrm{Bernoulli}}

\global\long\def\bernDist{\mathrm{Bern}}

\global\long\def\binDist{\mathrm{Bin}}

\global\long\def\uniDist{\mathrm{Uni}}

\global\long\def\poiDist{\mathrm{Poi}}

\global\long\def\gammaDist{\mathrm{Gamma}}

\global\long\def\PP{\Pi}

\global\long\def\PPnu{\Pi_{\nu}}

\global\long\def\vertexset#1{v\left(#1\right)}

\global\long\def\edgeset#1{e\left(#1\right)}

\global\long\def\PPbelowth#1{\Pi_{\nu,\le#1}}

\global\long\def\PPaboveth#1{\Pi_{\nu,>#1}}

\global\long\def\linverse{^{-1}}

\global\long\def\threshold{T_{\nu}}

\global\long\def\popthreshold{T_{\nu,\mbox{pop}}}

\global\long\def\popgraph{P_{\nu}}

\global\long\def\upperthreshold{T_{\nu,u}}

\global\long\def\numDegNu#1{N_{\nu,#1}}

\global\long\def\PP{\Pi}

\global\long\def\PPnu{\Pi_{\nu}}

\global\long\def\vertexset#1{v\left(#1\right)}

\global\long\def\edgeset#1{e\left(#1\right)}

\global\long\def\PPbelowth#1{\Pi_{\nu,\le#1}}

\global\long\def\PPaboveth#1{\Pi_{\nu,>#1}}

\global\long\def\linverse{^{-1}}

\global\long\def\threshold{T_{\nu}}

\global\long\def\popthreshold{T_{\nu,\mbox{pop}}}

\global\long\def\popgraph{P_{\nu}}

\global\long\def\upperthreshold{T_{\nu,u}}

\global\long\def\numDegNu#1{N_{\nu,#1}}

\global\long\def\pointspace{\mathbb{M}}

\global\long\def\degNuFn{D_{\nu}}

\global\long\def\law{\mathcal{L}}

\newcommand{\dProk}[2]{d_{\mathrm{p}}(#1, #2)}

\newcommand{\dTV}[2]{\| #1 - #2 \|_{\mathrm{TV}}}

\newcommand{\Lebesgue}{\Lambda}

\newcommand{\NatSubs}[1]{\tilde \Nats_{#1}}

\newcommand{\StarF}{S} \newcommand{\IsoF}{I}

\newcommand{\convEst}{\to_{\mathrm{est}}}

\newcommand{\convEstGP}{\to_{\mathrm{GP}}}

\newcommand{\convEstGS}{\to_{\mathrm{GS}}}

\newcommand{\ExtNats}{\overline{\Nats}}

\newcommand{\W}{\mathcal{W}}

\newcommand{\GP}{\mathcal{G}}

\newcommand{\GS}[1]{\mathscr G(#1)}

\newcommand{\KEG}{\Gamma}

\newcommand{\GPD}[2]{\mathrm{uKEG}(#1,#2)}

\newcommand\KEGD[2]{\KEGDinf{#1,#2}}
\newcommand\KEGDinf[1]{\mathrm{KEG}(#1)}

\newcommand{\ICP}[1]{[#1]}
\newcommand{\RGraphs}{\mathcal G_{\Reals}}
\newcommand{\URGraphs}{{\RGraphs^{\sim}}}
\newcommand{\URGraphSeqs}{\URGraphs^{\infty}}

\newcommand{\CNGraphs}{\overline{\mathcal G}_{\Nats}}
\newcommand{\GraphSeqs}{\NGraphs^*}
\newcommand{\allreps}{\mathsf{adjmatrices}}

\newcommand{\CAM}[1]{\mathcal G(#1)}

\newcommand{\edges}{e}

\newcommand{\vertices}{v}

\newcommand{\empGraphon}{\tilde{W}}

\newcommand{\kegname}{graphex process}
\newcommand{\kegnames}{graphex processes}

\title{Sampling and estimation for (sparse) exchangeable graphs}
  
 \author[V.~Veitch]{Victor Veitch} 
 \address{University of
 Toronto\\Department of Statistical Sciences\\Sidney Smith Hall\\100
 St George Street\\Toronto, Ontario\\M5S 3G3\\Canada}

 \author[D.~M.~Roy]{Daniel M.~Roy}
\address{University of
 Toronto\\Department of Statistical Sciences\\Sidney Smith Hall\\100
 St George Street\\Toronto, Ontario\\M5S 3G3\\Canada}

\begin{abstract}
Sparse exchangeable graphs on $\mathbb{R}_+$, and the associated graphex framework for sparse graphs,
generalize exchangeable graphs on $\mathbb{N}$, and the associated graphon framework for dense graphs.
We develop the graphex framework as a tool for statistical network analysis
by identifying the sampling scheme that is naturally associated with the models of the framework,
and by introducing a general consistent estimator for the parameter (the graphex) underlying these models.
The sampling scheme is a modification of independent vertex sampling that throws away vertices that are isolated in the sampled subgraph.
The estimator is a dilation of the empirical graphon estimator, which is known to be a consistent estimator for dense exchangeable graphs; both can be understood as graph analogues to the empirical distribution in the i.i.d.\ sequence setting. 
Our results may be viewed as a generalization of consistent estimation via the empirical graphon from the dense graph regime to also include sparse graphs.

\end{abstract}

\maketitle

\begin{center}
  \begin{minipage}{.80\linewidth}
    \setcounter{tocdepth}{1}
    \tableofcontents
  \end{minipage}
\end{center}

\section{Introduction}

This paper is concerned with 
mathematical foundations for 
the statistical analysis of real-world networks.
For densely connected networks, 
the graphon framework has emerged as powerful tool for both theory and applications in network analysis;
many of the models used in practice are within the remit of this framework (see \citep{Orbanz:Roy:2015} for a review).
However, in most real-world situations, networks are sparsely connected; i.e., as one studies larger networks, one finds that they tend to exhibit only a vanishing fraction of all possible links.
 
In this paper, we continue our study of sparse exchangeable graphs, i.e., random graphs
whose vertices may be identified with nonnegative reals, $\mathbb{R}_+$, and  
whose edge sets are then modeled by exchangeable point processes on $\mathbb{R}_+^2$.
In a pioneering paper, \citet{Caron:Fox:2014} introduced the notion of sparse exchangeable graphs in the context of nonparametric Bayesian analysis.
Building on this work,
the general family of all sparse exchangeable graphs was characterized by \citet{Veitch:Roy:2015,Borgs:Chayes:Cohn:Holden:2016}, and shown to generalize the graphon models for dense graphs to include the sparse graph regime.
Sparse exchangeable graphs have a number of desirable properties, including that they define a natural projective family of subgraphs of growing size, which can be used to model the process of observing a larger and larger fraction of a fixed underlying network. This property also provides a firm foundation for the study of various asymptotics, as  demonstrated by \citep{Veitch:Roy:2015, Borgs:Chayes:Cohn:Holden:2016} and the results here.
\citep{Veitch:Roy:2015} characterized the asymptotic degree distribution and connectedness of sparse exchangeable graphs, demonstrating that sparse exchangeable graphs allow for sparsity and admit the rich graph structure (such as small-world connectivity and power law degree distributions) found in large real-world networks.  On this basis, \citeauthor{Veitch:Roy:2015} argue that sparse exchangeable graphs can serve as a general statistical model for network data.  Despite our understanding of these models, the statistical meaning remains somewhat opaque. Put simply, when would it be natural to use the sparse exchangeable graph model?

The present paper further develops this framework for statistical network analysis by 
answering two fundamental questions:
\begin{enumerate}
\item What is the notion of sampling naturally associated with this statistical network model? and
\item How can we use an observed dataset to consistently estimate the statistical network model? 
\end{enumerate}
The answers to these questions significantly clarify both the meaning of the modeling framework,
and its connection to the dense graph framework and the classical i.i.d.\ sequence framework at the foundation of classical statistics.
These questions may be viewed as specific examples of a general approach to formalizing the problem of statistical analysis 
on network data being carried out by Orbanz~\citep{Orbanz:2016}.

To explain the results, we first recall the modeling framework of \citep{Veitch:Roy:2015,Borgs:Chayes:Cohn:Holden:2016}.
The basic setup introduces a family of finite, symmetric point processes $\KEG_s \subseteq [0,s] \times [0,s]$, for $s \in \NNReals$,
where each $\KEG_s$ is interpreted as the edge set of a random graph whose vertices are points in the interval $[0,s]$.
Hence, for $\theta,\theta' \in [0,s]$, there is an edge between $\theta$ and $\theta'$ if and only if  $(\theta, \theta') \in \KEG_s$.
The edge set $\KEG_s$ determines a graph over its active vertex set: those elements $\theta \in [0,s]$ such that $\theta$ exhibits some edge in $\KEG_s$.
Accordingly, $(\KEG_s)_{s\in\NNReals}$ are understood as ($\NNReals$-labeled) graph-valued random variables
that are nested in the sense that $\KEG_r \subseteq \KEG_s $ whenever $r \le s$.
We will argue below that the indices $s \in \NNReals$ are properly understood as specifying the sample \emph{size} of the corresponding observations $\KEG_s$.  

\newcommand{\bvtheta}{\boldsymbol{\vartheta}}
The natural parameter of the distributions of these graphs is a \emph{graphex} $\W = (I,S,W)$ 
defined on some locally finite measure space $(\bvtheta,\mathcal B_{\bvtheta},\nu)$
where 
$I \in \NNReals$, $S: \bvtheta \to \NNReals$ is an integrable function, and
$W: \bvtheta^2 \to [0,1]$ is a symmetric function satisfying several weak integrability conditions we formalize later.
(Without loss of generality, one can always take $(\bvtheta,\mathcal B_{\bvtheta})$ to be the non-negative reals, $\NNReals$, with its standard Borel structure, and take $\nu$ to be Lebesgue measure $\Lebesgue$.)
The component $W$ is a natural generalization of the graphon of the dense graph models \citep{Veitch:Roy:2015,Borgs:Chayes:Cohn:Holden:2016}, 
and for this reason we refer to it as a \emph{graphon}.
Although the results of the present paper hold for general graphexes, for simplicity of exposition, we 
will temporarily restrict our attention to graphexes of the form $\W=(0,0,W)$, giving a full treatment in subsequent sections. 

We begin by giving a construction of the {\kegname} for a graphex of the form $\W=(0,0,W)$:
Let
$\PP$ be a Poisson (point) process on $\NNReals \times \bvtheta$ with intensity $\Lebesgue \otimes \nu$, i.e., for two intervals $J_1,J_2 \subseteq \NNReals$ and two measurable subsets $B_1,B_2 \subseteq \bvtheta$, the number of points of $\PP$ in $J_1 \times B_1$ and in $J_2 \times B_2$ are Poisson random variables with mean $|J_1| \nu(B_1)$ and $|J_2| \nu(B_2)$ respectively, where $|J_i| = \Lebesgue(J_i)$ is the length of the interval, and these variables are even independent when $J_1\times B_1 \cap J_2 \times B_2 = \emptyset$. (If $\nu$ is also Lebesgue measure on $\NNReals$, then $\PP$ is then simply a unit-rate Poisson process on $\NNReals^2$.)
Write $\{(\theta_{i},\vartheta_{i})\}_{i\in \Nats}$ for the points of $\PP$, 
let $\PP_s$ be the restriction of $\PP$ to $[0,s]^2$, 
and let $(\zeta_{\{i,j\}})_{i \le j \in \Nats}$ be an i.i.d.\ collection of uniform random variables in $[0,1]$. 

For every $s \in \NNReals$, define the size-$s$ random edge set $\KEG_s$ on $[0,s]$
to be exactly the set of distinct pairs $(\theta_i,\theta_j)$ where $\theta_i,\theta_j \le s$
and $\zeta_{\{i,j\}} \le W(\vartheta_i,\vartheta_j)$.  In other words, 
for every distinct pair of points $(\theta,\vartheta),(\theta',\vartheta') \in \PP_s$,
the edge set $\KEG_s$ includes the edge $(\theta,\theta')$
independently with probability $W(\vartheta,\vartheta')$.
The vertex set of the graph corresponding to the edge set $\KEG_s$ is defined to be those points that appear in some edge; hence, this model does not allow for isolated vertices.
The entire family of graphs $\KEG_s$, for $s \in \NNReals$, is a projective family with respect to subset restriction, i.e., $\KEG_r = \KEG_s \cap [0,r]^2$ for every $r,s \in \NNReals$ with $r \le s$.
See \cref{fig:gen_proc} for an illustration of the generative model for a general graphex $\W$ defined on $\NNReals$ with Lebesgue measure.  A rigorous definition is provided in \cref{prelim}.

We refer to $\KEG$ as the \emph{{\kegname}} generated by $\W$: 
we also use this nomenclature for the family $(\KEG_s)_{s\in\NNReals}$.
Note that the $\KEG$ has the property that its distribution is invariant to the action of the maps $(x,y) \mapsto (\phi(x),\phi(y))$, where $\phi : \NNReals \to \NNReals$ is measure preserving.  A random graph with this property is called a sparse exchangeable graph \citep{Caron:Fox:2014}.  In \cref{prelim}, we quote the result due to \citet{Veitch:Roy:2015,Borgs:Chayes:Cohn:Holden:2016}, building off work by \citet{Kallenberg_Random_Meas_Plane}, that proves that every (sparse) exchangeable graph is a {\kegname} generated by some (potentially random) graphex.
For a finite labeled graph $G$, such as each $\KEG_s$, for $s \in \NNReals$, 
we will write $\CAM{G}$ to denote the \emph{unlabeled}\footnote{The unlabelled graph corresponding to a labelled graph $G$ is the equivalence class of graphs isomorphic to $G$. Restricting ourselves to finite unlabelled graphs, we can represent the unlabelled graphs formally in terms of their homomorphism counts, $(N_F)$, where $F$ ranges over the countable set of all finite simple graphs whose vertex set is $[n]$ for some $n \in \Nats$, and $N_F$ is the number of homomorphisms from $F$ to $G$.}
 graph corresponding to $G$.

\begin{figure}
 \includegraphics[width=0.99\linewidth]{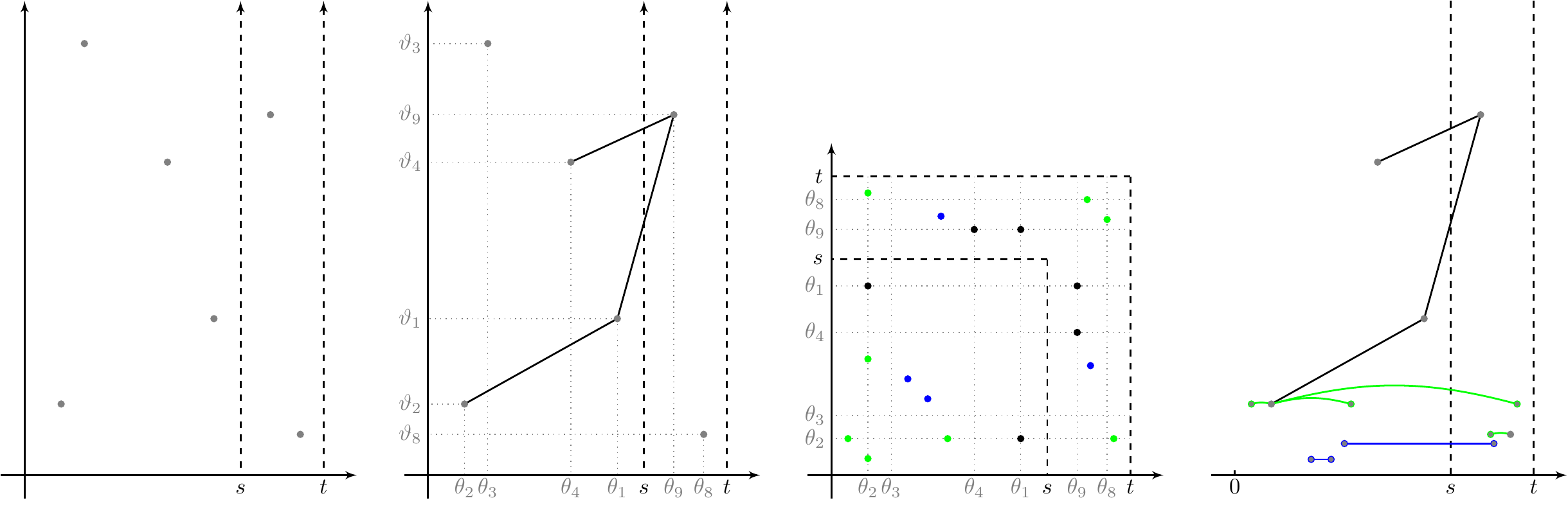}
\caption{\label{fig:gen_proc}
  Generative process of a {\kegname} generated by a graphex $\W=(I,S,W)$ defined on $\NNReals$ with Lebesgue measure, 
  observed at sizes $s$ and $t$.
  First panel: a (necessarily truncated) realization of the latent Poisson process $\PP_t$ on $[0,t] \times \NNReals$. A countably infinite number of points lie above the six points visualized.
  Second panel: Edges due to the graphon component $W$ are sampled by connecting each distinct pair of points $(\theta_i,\vartheta_i), (\theta_j, \vartheta_j) \in \PP_t$ 
  independently with probability $W(\vartheta_i, \vartheta_j)$. 
  Integrability conditions on $W$ imply that only a finite number of edges will appear, despite there being an infinite number of points in $\PP_t$. Assume the three edges are the only ones.
  Third panel: The edge set $\KEG_t$ represented as an adjacency measure on $[0,t]^2$.  
  The edges in the graphon component appears as (symmetric pairs of) black dots; the edges corresponding to the star component $S$ appear in green; the isolated edges (from the $I$ component) appear in blue. 
  At size $s$, only the edges in $[0,s]^2$ (inner dashed black line) appear in the graph.
  The edges $\{\theta_j, \sigma_{jk}\}$ of the star ($S$) component of the process (green) centered at $\theta_j$ are
  realizations of a rate-$S(\vartheta_j)$ Poisson process $\{\sigma_{jk}\}$ along the line through $\theta_j$ (show as green dots along grey dotted lines).
  Hence, at size $t$, each point $\theta_i$ is the center of $\poiDist(t\,S(\vartheta_i))$ star process rays.
  The edges $\{\rho_i, \rho'_j\}$ generated by the isolated edge ($I$) component of the process (blue)
  are a realization of a rate-$I$ Poisson process on the upper (or lower) triangle of $[0,t]^2$, reflected.
  At size $t$, there are $\poiDist(t^2\,I)$ isolated edges due to this part of the graphex. 
  The final panel shows the graphs corresponding to the sampled adjacency measure at sizes $s$ and $t$. 
}
\end{figure}

The first contribution of the present paper is the identification of 
a sampling scheme that is naturally associated with the {\kegnames}:
\begin{defn}
A \defnphrase{$p$-sampling} of an unlabeled graph $G$ is obtained by 
selecting each vertex of $G$ independently with probability $p \in [0,1]$,
and then returning the edge set of the random vertex-induced subgraph of $G$.
\end{defn}
It is important to note that only the edge set of the vertex-induced subgraph is returned; 
in other words, vertices that are isolated from the other sampled vertices are thrown away.
The key fact about this sampling scheme is that:
For $s > 0$ and $r \in [0,s]$, if $G_r$ is an $r/s$-sampling of $\CAM{\KEG_s}$ then $G_r \equaldist \CAM{\KEG_r}$.
This result justifies the interpretation of the parameter $s$ as a sample size. 

In the estimation problem for the \kegname, 
the observed dataset is a realization of the random sequence of graphs $G_1, G_2, \dots$ such that
$G_k = \CAM{\KEG_{s_k}}$, and $s_1, s_2, \dots$
is some sequence of sizes such that $s_k \upto \infty$ as $k \to \infty$.
The task is to take such an observation and return an estimate for $\W$,
where $\W$ is the graphex that generated $(\KEG_s)_{s\in\NNReals}$.  
Both the formulation and solution of this problem depend on whether the sizes $s_k$ are included 
as part of the observations.

We first treat the simpler case where the sizes are known.
To formalize the estimation problem we must introduce a
notion of when one graphex is a good approximation for another. 
Intuitively, our notion is that, for any fixed $s$, a size-$s$ random graph generated by
an estimator should be close in distribution to a size-$s$ random graph generated by the true graphex.
Let $\GPD{\W}{s}$ be the distribution of an unlabeled size-$s$ \kegname, i.e., the distribution of $\CAM{\KEG_s}$ where $\KEG$ is generated by $\W$.
Approximation is then formalized by the following notion of convergence:
\begin{defn}
Write $\W_k \convEstGP \W$ as $k \to \infty$, when $\GPD{\W_k}{s} \to \GPD{\W}{s}$ weakly as $k \to \infty$, for all $s \in \NNReals$.
\end{defn}
Our goal in the estimation problem is then to take a sequence of observations and use these to produce a sequence of graphexes $\W_1, \W_2, \dots$ 
that are consistent in the sense that $\W_k \convEstGP \W$ as $k \to \infty$. 
This is a natural analogue of the definition of consistent estimation used for the convergence of the empirical cumulative distribution function
in the i.i.d. sequence setting, and of the definition of consistent estimation used for the convergence of the empirical graphon in the dense graph setting.

Let $\vertices(G)$ denote the number of vertices of graph $G$.
Our estimator is the \emph{dilated empirical graphon}
\[
\hat{W}_{(G_k, s_k)} : [0,\vertices(G_k)/s_k)^2 \to \{0,1\},
\] 
defined by transforming
the adjacency matrix of $G_k$ into a step function on $[0,\vertices(G_k)/s_k)^2$
where each pixel has size $1/s_k \times 1/s_k$; see \cref{dilated_emp_graphon}.
Intuitively, when the generating graphex is $\W=(0,0,W)$, we have $s_k \upto \infty$ as $k \to \infty$, 
and the estimator is an increasingly higher and higher resolution pixel picture of the generating graphon. 
Formally, given a non-empty finite graph $G$ with $n$ vertices labeled $1, \dots, n$,
we define the empirical graphon $\empGraphon_{G}:[0,1]^2 \to \{0,1\}$ by partitioning $[0,1]$ into adjacent intervals $I_1, \dots, I_n$
each of length $1/n$ and taking $\empGraphon_{G} = 1$ on $I_i \times I_j$ if $i$ and $j$ are connected in $G$, and taking $\empGraphon = 0$ otherwise.
The dilated empirical graphon with dilation $s$ is then defined by $\hat{W}_{(G, s)}(x,y) = \empGraphon_{G}(x/s,y/s)$.
To map an unlabeled graph to a (dilated) empirical graphon we must introduce a labeling of the vertices.
Notice that
if $\phi:\NNReals\to\NNReals$ is a measure-preserving transformation, $\phi \otimes \phi$ is 
the map $(\phi \otimes \phi) (x,y) = (\phi(x),\phi(y))$, $\W = (I,S,W)$, and 
$\W' = (I, S \circ \phi, W \circ (\phi \otimes \phi))$, 
then $\GPD{\W}{s} = \GPD{\W'}{s}$ for all $s \in \NNReals$.
In particular, the dilated empirical graphon functions corresponding to different labelings of the vertices of $G$ 
are related by obvious measure-preserving transformations in this way. 
For the purposes of this paper, graphexes that give rise to the same distributions over graphs are equivalent.
We then define the empirical graphon of an unlabeled graph to be the empirical graphon of that graph with some
arbitrary labeling, and we define the dilated empirical graphon similarly. 
These functions may be thought as arbitrary representatives of the equivalence class on graphons given by 
equating two graphons whenever they correspond to isomorphic graphs.

\begin{figure}
 \includegraphics[width=0.99\linewidth]{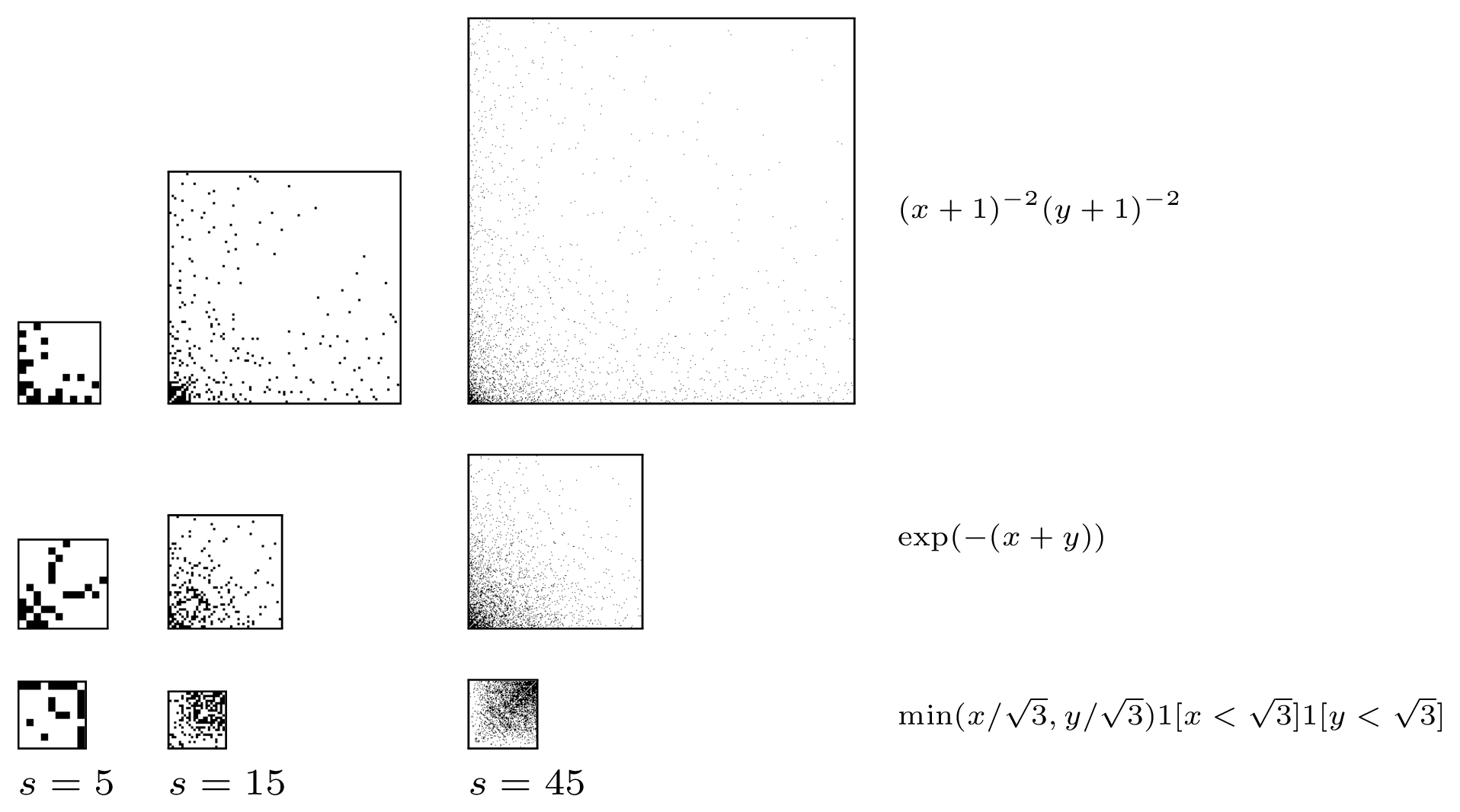}
\caption{\label{dilated_emp_graphon}
  Realizations of dilated empirical graphons of {\kegnames} generated by $(0,0,W)$ for $W$ given in the rightmost column,
  at observation sizes given in the bottom row.
  Note that the ordering of the vertices used to define the estimator is arbitrary.
  Here we have suggestively ordered the vertices according to the latent values from the process simulations;
  with this ordering the dilated empirical graphons are approximate pixel pictures of the generating graphon
  where the resolution becomes finer as the observation size grows.
  All three graphons satisfy $\|W\|_1 = 1$, and thus the expected number of edges (black pixels) at each size $s$ is $\frac{1}{2}s^2$ in each column.
  Note that the rate of dilation is faster for sparser graphs; 
  as established in \cite{Veitch:Roy:2015}, the topmost {\kegname} used for this example
  is sparser than the middle \kegname, and the graphon generating the bottom {\kegname} is compactly supported and thus corresponds to a dense graph. 
}
\end{figure}

The first main estimation result is that $\hat{W}_{(G_k, s_k)} \convEstGP W$ in probability as $k \to \infty$.
That is, for every infinite sequence $N \subseteq \Nats$ there is a further
infinite subsequence $N' \subseteq N$ such that $\hat{W}_{(G_k, s_k)} \convEstGP W$ almost surely along $N'$.
Subject to an additional technical constraint (implied by integrability of $W$) the convergence in probability
may be replaced by convergence almost surely.
Note that consistency holds for observations generated by an arbitrary graphex $\W = (I,S,W)$,
not just those with the form $\W=(0,0,W)$; see \cref{fig:triple_proc}.

\begin{figure}
 \includegraphics[width=0.99\linewidth]{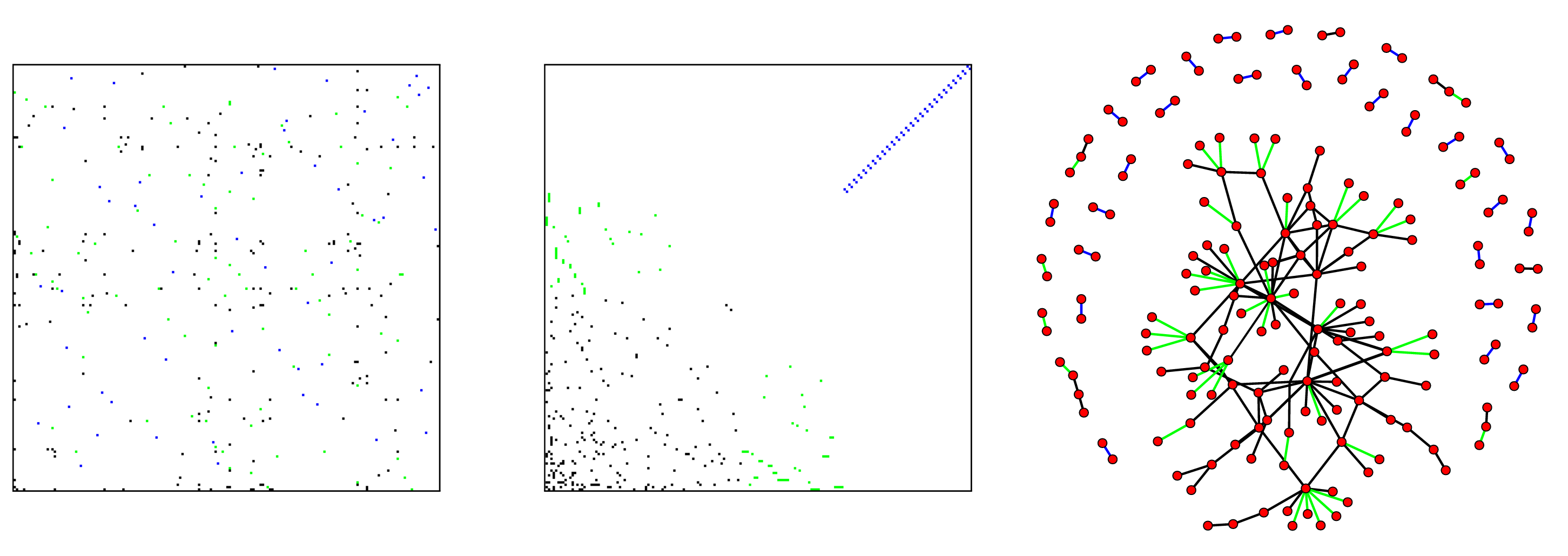}
\caption{\label{fig:triple_proc}
Realization of unlabeled {\kegname} generated by $\W=(I,S,W)$ at size $s=15$ (right panel), and associated dilated empirical graphon (left and center panels).
The generating graphex is $W = (x+1)^{-2}(y+1)^{-2}$, $S = 1/2 \exp(-(x+1))$, and $I = 0.1$.
The observation size is $s=15$.
The dilated empirical graphex is pictured as two equivalent representations $\hat{W}_{(G,15)}$ and $\hat{W}'_{(G,15)}$, each with support $[0,12)^2$ ($180$ vertices at size $15$).
Edges from the $W$ component are shown in black, edges from the $S$ component are shown in green, and
edges from the $I$ component are shown in blue.
Recall that the ordering of the dilated empirical graphon is arbitrary, so the left and center panels depict 
different representations of the same estimator.
The leftmost panel shows the dilated empirical graphon with a random ordering.
The middle panel shows the dilated empirical graphon sorted to group the $I$, $S$, and $W$ edges,
with the $W$ edges sorted as in \cref{dilated_emp_graphon}.
The middle panel gives some intuition for why the dilated empirical graphon is able to estimate the entire graphex triple:
When a {\kegname} is generated according to $\hat{W}_{(G,15)}$ with latent Poisson process $\PP$,
the disjoint structure of the dilated graphon regions due to the $I$, $S$, and $W$ components induces a natural
partitioning of $\PP$ into independent Poisson processes that reproduce the independence structure used
in the full generative model \cref{KEGgen}.
}
\end{figure}

We now turn to the setting where the observation sizes $s_1, s_2, \dots$ are not included as part of the observations.
In this case, we study two natural models for the dataset. The first is to treat the observed
graphs $G_k$ as realizations of $\CAM{\KEG_{s_k}}$ for some (unknown) sequence $s_k \upto \infty$ as $k \to \infty$
that is independent of $\KEG$.
Another natural model is to take $G_1, G_2, \dots$ to be the sequence of all distinct graph structures taken on by $(\KEG_s)_{s\in\NNReals}$;
in this case, for all $k$, we take $G_k = \CAM{\KEG_{\tau_k}}$, where $\tau_k$ is the latent size at the $k$th occasion that the graph structure changes.
In this later case, we call $\GS{\KEG} = (\CAM{\KEG_{\tau_1}},\CAM{\KEG_{\tau_2}}, \dots )$ the graph sequence of $\KEG$. (We define the graph sequence formally in \cref{prelim}.)

Intuitively, $\GS{\KEG}$ is the {\kegname} $\KEG$ with the size information stripped away.
In this sense, the graph sequence of $\KEG$ is the random object naturally associated to $\W$ when the sizes are unobserved.
Thus, in this setting, 
convergence in distribution of the graph sequences induced by the estimators is a natural notion of consistency.
\begin{defn}
Write $\W_k \convEstGS \W$ as $k \to \infty$ when $\GS{\KEG^k} \convDist \GS{\KEG}$ as $k \to \infty$,
for $\KEG^k$ generated by $\W_k$ and $\KEG$ generated by $\W$.
\end{defn}
The notion of consistent estimation corresponding to this convergence is that, 
for any fixed $\ell \in \Nats$, the distribution of the length $\ell$ prefix of the graph sequence generated by the estimator 
should be close to the distribution of the length $\ell$ prefix of the graph sequence generated by $\W$.
Convergence in distribution of every finite-size prefix is equivalent to convergence in distribution of
the entire sequence. 

To explain our estimator for this setting, we will need the following concept:
\begin{defn}
Let $c \in \NNReals$ and let $\W = (I,S,W)$ be a graphex. A \defnphrase{$c$-dilation of $\W$} is
the graphex  $\W^c =(c^2I , c S (\cdot / c) , W (\cdot / c, \cdot / c))$.
\end{defn}
The key fact about $c$-dilations is that $\GPD{\W}{s} = \GPD{\W^c}{s/c}$ for all $s \in \NNReals$,
and thus also $\GS{\KEG} \equaldist \GS{\KEG^c}$ whenever $\KEG$ is generated by $\W$ and $\KEG^c$ is generated by $\W^c$.
That is, the law of the graph sequence is invariant to dilations of the generating graphex.
This means, in particular, that the dilation of a graphex is not an identifiable parameter
when the observation sizes are not included as part of the observation. 
The obvious guess for the estimator in this setting is then the estimator for the known-sizes setting
with the dilation information stripped away.
That is, our estimator is the dilated empirical graphon modulo dilation;
i.e., it is simply the empirical graphon $\empGraphon_{G_k} : [0,1]^2 \to [0,1]$ defined above.
In this setting, the empirical graphon is acting as a representative of its equivalence class under 
the relation that equates graphons that generate graph sequences with the same laws.

The main estimation result is that if either
\begin{enumerate}
\item There is some (possibly random) sequence $(s_k)$, independent from $\KEG$, such that
$s_k \upto \infty$ a.s.\ and $G_k = \CAM{\KEG_{s_k}}$ for all $k \in \Nats$, or
\item $(G_1, G_2, \dots) = \GS{\KEG}$,
\end{enumerate}
then $\empGraphon_{G_k} \convEstGS \W$ in probability as $k \to \infty$.
Subject to an additional technical constraint (implied by integrability), 
the convergence in probability may be strengthened to convergence almost surely.

Our estimation results are inspired by Kallenberg's development of the theory
of estimation for exchangeable arrays \citep{Kallenberg:1999}.
Restricted to the graph setting (that is, $2$-dimensional arrays interpreted as adjacency matrices), 
and translated into modern language,
that paper introduced the empirical graphon (although not named as such) and formalized consistency in terms of the weak topology: $W_k \to W$ as $k \to \infty$ when
the graphs generated by $W_k$ converge in distribution to the graphs generated by $W$.
The estimation results of the present paper may be seen as generalizations of \citep{Kallenberg:1999} to the sparse graph regime. 

The present paper is also closely related to the recent paper \citep{Borgs:Chayes:Cohn:Holden:2016}.
Specialized to the case $\bvtheta = \NNReals$ equipped with Lebesgue measure, that paper extends the 
cut distance between compactly supported graphons---a core tool in the limit theory of dense graphs---to arbitrary integrable graphons.
Convergence in the cut distance then gives a notion of limit for sequences of graphons.
This is extended to a notion of convergence for sequences of (sparse) graphs by
saying that a sequence $G_1, G_2, \dots$ converges in the stretched cut distance sense
if and only if $\hat{W}_{(G_1, \sqrt{\edges(G_1})}, \hat{W}_{(G_2, \sqrt{\edges(G_2})}, \dots$ converges with respect to the cut distance.
That is, each graph $G_k$ is mapped to the empirical graphon dilated by $\vertices(G_k)/\sqrt{\edges(G_k)}$. 
The same paper also establishes that $\edges(G_k)/s_k^2 \to \|W\|_1 \as$.
Thus, in the $\|W\|_1 = 1$ case, these dilated empirical graphons, considered as pixel pictures, will look asymptotically identical
to the $\vertices(G_k)/s_k$-dilated empirical graphons that we use as estimators in the known sizes case. 
This suggests that a close connection between consistent estimation and convergence in the cut distance.
Indeed, in the dense graph setting these notions of convergence are known to be equivalent 
(in the dense setting, the convergence $W_k \convEstGP W$ as $k \to \infty$ is equivalent to left convergence \citep{Diaconis:Janson:2007},
and left convergence is equivalence to convergence in the cut norm \citep{Borgs:Chayes:Lovasz:Sos:Vesztergombi:2006}).
An analogous result in the sparse graph setting would allow for a very different approach to proving our
convergence result in the known size setting, restricted to the special case that the generating graphex is an integrable graphon.

The paper is organized as follows:
In \cref{prelim} we give formal definitions for the basic tools of the paper.
The sampling result is derived in \cref{sampling}.
In \cref{est_w_sizes} we prove the estimation result for the setting where observation sizes are included as part of the observation.
We build on this in \cref{est_wo_sizes} to prove the estimation result for the setting where the true underlying observation sizes are not observed.

\newcommand{\NGraphs}{\mathcal G_{\Nats}}
\newcommand{\AM}[1]{G(#1)}

\newcommand{\IC}[1]{[\![#1]\!]}

\section{Preliminaries}
\label{prelim}
The basic object of interest in this paper is point processes on $\NNReals^2$,
interpreted as the edge sets of random graphs with vertices labeled in $\NNReals$.
\begin{defn}
  An \defnphrase{adjacency measure} is a purely atomic, symmetric, simple, locally finite measure on $\NNReals^2$.
\end{defn}
If $\xi = \sum_{i,j}\delta_{(\theta_i,\theta_j)}$ is an adjacency measure then the associated graph with labels in $\NNReals$  
is one with edge set $\{(\theta_i,\theta_j)\}$, where $\theta_i \le \theta_j$; the vertex set is deduced from the edge set.

The defining property of {\kegnames} is that, intuitively speaking,
the labels of the vertices of the graph are uninformative about the graph structure.
This is formalized by requiring that the associated adjacency measure is jointly exchangeable, where
\begin{defn}
  A random measure $\xi$ on $\NNReals^2$ is \defnphrase{jointly exchangeable} if $\xi \circ (\phi \otimes \phi) \equaldist \xi$
  for any measure-preserving transformation $\phi:\NNReals \to \NNReals$.
\end{defn}    

A representation theorem for jointly exchangeable random measures on $\NNReals^2$ was given by Kallenberg \citep{Kallenberg:2005, Kallenberg_Random_Meas_Plane}. 
This result was translated to the setting of random graphs in \citep{Veitch:Roy:2015}.
Writing $\Lebesgue$ for Lebesgue measure and $\mu_W(\cdot) = \int_{\NNReals} W(x,\cdot) \intd x$,
the defining object of the representation theorem is:
\begin{defn}
  A \defnphrase{graphex} is a triple $(\IsoF,\StarF,W)$, where $\IsoF \ge 0$ is a non-negative real,
  $\StarF : \NNReals \to \NNReals$ is integrable,
  and the \defnphrase{graphon} $W : \NNReals^2 \to [0,1]$ is symmetric, and satisfies
  \begin{enumerate}
  \item $\Lebesgue\{ \mu_W = \infty \} = 0$ and $\Lebesgue \{\mu_W > 1\} < \infty$, 
  \item $\Lebesgue^2[W; \mu_W \vee \mu_W \le 1] = \int_{\NNReals^2} W(x,y)\,1[ \mu_W(x) \le 1 ]\, 1[\mu_W(y) \le 1] \intd x \intd y < \infty$,
  \item $\int_{\NNReals} W(x,x) \, \intd x < \infty$.
  \end{enumerate}
  We say that a graphex is non-trivial if $I + \|S\|_1 + \|W\|_1 > 0$, i.e. if it is not the case that the graphex is $0$ a.e.
\end{defn}
The representation theorem is:
\begin{thm}\label{thm:graphex_rep_theorem}
  Let $\xi$ be a random adjacency measure.
  $\xi$ is jointly exchangeable \iff there exists a (possibly random) graphex $\W=(\IsoF,\StarF,W)$ such that, almost surely, 
  \begin{align}
  \begin{split}\label{KEGgen}
    \xi = 
    & \hspace*{3.6mm} 
            \sum_{i,j}1[W(\vartheta_{i},\vartheta_{j}) \le \zeta_{\{i,j\}}]\delta_{\theta_{i},\theta_{j}}\\
    & +  \sum_{j,k} 1[ \chi_{jk} \le \StarF(\vartheta_{j}) ](\delta_{\theta_{j},\sigma_{jk}}+\delta_{\sigma_{jk},\theta_{j}})\\
    & +   \sum_{k} 1[\eta_{k} \le \IsoF ](\delta_{\rho_{k},\rho'_{k}}+\delta_{\rho'_{k},\rho_{k}}),
  \end{split}
  \end{align}
  for some collection of independent uniformly distributed random variables $(\zeta_{\{i,j\}})$ in $[0,1]$;
  some independent unit-rate Poisson processes $\{ (\theta_{j},\vartheta_{j})\} $ and
  $\{ (\sigma_{ij},\chi_{ij})\} _{j}$, for $i\in\Nats$, on $\NNReals^{2}$
  and $\{ (\rho_{j},\rho_{j}^{\prime},\eta_{j})\} $ on $\NNReals^{3}$.
\end{thm}

\begin{defn}
  A \defnphrase{\kegname} associated with
  graphex $(\IsoF,\StarF,W)$ is the random adjacency measure $\KEG$ of the form given in \cref{KEGgen}.
  The {\kegname} model is the family $(\KEG_s)_{s\in\NNReals}$, where $\KEG_s(\cdot) = \KEG( \cdot \, \cap [0,s]^2)$.
\end{defn}
\begin{remark}
In \citep{Veitch:Roy:2015} the Kallenberg exchangeable graph was defined as the random graph with vertex labels in $\NNReals$ associated with $\KEG$.
The definition of the {\kegname} differs slightly, motivated by the use of techniques from the theory of distributional convergence of point processes, which makes explicit appeal to the point process structure desirable.
It will sometimes be useful in exposition to conflate the {\kegname} with the associated labeled graph, so statements such as ``the number of edges of $\KEG_s$'' are sensible.    
\end{remark} 

We will often have occasion to refer to the unlabeled finite graph associated with a finite adjacency measure.  
\begin{defn}
Let $\xi$ be a finite adjacency measure. 
The \defnphrase{unlabelled graph associated with $\xi$} is $\CAM{\xi}$.
\end{defn}

A particularly important case is the graph associated to the size-$s$ {\kegname} $\KEG_s$, which is almost surely finite.
We will have frequent occasion to refer to the distributions of both the labeled and unlabeled graphs:
\begin{defn}
Let $(\KEG_s)_{s\in\NNReals}$ be a graphex process generated by $\W$.
The \defnphrase{finite {\kegname} distribution} with parameters $\W$ and $s$ is $\KEGD{\W}{s} = \Pr(\KEG_s \in \cdot \given \W,s)$,
and $\KEGDinf{\W} = \KEGD{\W}{\infty}$.
The \defnphrase{finite unlabeled {\kegname} distribution} with parameters $\W$ and $s$ is $\GPD{\W}{s} = \Pr(\CAM{\KEG_s} \in \cdot \given \W,s)$.
\end{defn}

In order to pass from $\CAM{\xi}$ back to some adjacency measure $\xi'$ such that $\CAM{\xi'} = \CAM{\xi}$, we must reintroduce labels.
A simple scheme is to produce labels independently and uniformly in some range:
\newcommand{\RLabel}[2][]{\mathsf{Lbl}_{#1}(#2)}
\newcommand{\REmbed}[1]{\mathsf{embed}(#1)}
\begin{defn}
Let $G$ be an unlabeled graph with edge set $E$, and let $s > 0$.  
A \defnphrase{random labeling of $G$ into $[0,s]$}, $\RLabel[s]{G,\{U_i\}}$,
is a random adjacency measure $\RLabel[s]{G,\{U_i\}} = \sum_{(i,j) \in E} \delta_{(U_i,U_j)}$, 
where $U_i \distiid \uniDist[0,s]$, for $i \in \Nats$.
Where there is no risk of confusion, we will write $\RLabel[s]{G}$ for $\RLabel[s]{G, \{U_i\}}$ where $U_i \distiid \uniDist[0,t]$, for $i \in \Nats$,
independently of everything else. 
\end{defn}
Because our notion of consistent estimation is a requirement of distributional convergence, the distributions of these random labelings will play a large role.
Clearly, the distribution of $\RLabel[s]{G}$ is a measurable function of $G$ and $s$.
\begin{defn}
We write $\REmbed {G,s}(\cdot) = \Pr(\RLabel[s]{G} \in \cdot)$ for the distribution of $\RLabel[s]{G}$.
When $G$ is itself random, a random embedding of $G$ into $[0,s]$ is defined by $\REmbed {G,s} = \Pr [\RLabel[s]{G} | G]$.  
\end{defn}

We typically think of {\kegnames} as defining a nested collection of $\NNReals$-labeled 
graph valued random variables $(\KEG_s)_{s\in\NNReals}$. 
In modeling situations where the labeling is irrelevant, it is natural to instead look at the (countable) collection
of all distinct graph structures taken on by $(\KEG_s)_{s\in\NNReals}$; this is the graph sequence associated with $\KEG$.
We now turn to formally defining the graph sequence associated with an arbitrary adjacency measure $\xi$.
To that end, define $E : \NNReals \to \Nats$ by
\[
E(s) = \frac 1 2 \xi [0,s]^2 \qquad \text{for $s \in \NNReals$.}
\]
In the absence of self loops, $E(s)$ is the number of edges present
between vertices with labels in $[0,s]$.  
In general, the jumps of $E$ correspond with the appearance of edges.
\begin{defn}
Let $\xi$ be an adjacency measure.
The jump times of $\xi$, written as $\tau(\xi)$, is the sequence $\tau_1, \tau_2, \dotsc$ of jumps of $E$ in order of appearance.
\end{defn}
Note that the map $\xi \mapsto \tau(\xi)$ is measureable.
Intuitively, $\tau_1, \tau_2, \dots$ are the sample sizes at which 
edges are added to the unlabeled graph associated with the adjacency measure.

Let $\chi_s$ denote the operation of restricting an adjacency measure to those vertices with labels in $[0,s]$,
in the sense that $\chi_s\xi(\cdot) = \xi( \cdot \, \cap [0,s]^2)$.
We now formalize the sequence of all distinct unlabeled graphs associated with $(\chi_s\xi)_{s\in\NNReals}$:
\begin{defn}
The \defnphrase{graph sequence associated with $\xi$},
written $\GS {\xi}$, 
is the sequence $\CAM{\chi_{\tau_1}\xi}, \CAM{\chi_{\tau_2}\xi},\dotsc$,
where $\tau_1, \tau_2, \dotsc$ are the jump times of $\xi$.
\end{defn}

\section{Sampling}\label{sampling}
$\KEG_r$, a {\kegname} of size $r$,
may be generated from $\KEG_s$, a {\kegname} of size $s>r$, by restricting $\KEG_s$ to $[0,r]^2$.
In this section we show that this restriction has a natural relation to $p$-sampling:
$\CAM{\KEG_r}$ may be generated as an $r/s$-sampling of $\CAM{\KEG_s}$.

The first result we need is that 
random labelings preserve the law of exchangeable adjacency measures.
Intuitively, the labels 
of the size-$s$ {\kegname} can be invented by labeling each vertex \iid $\uniDist[0,s]$.

\begin{lem}\label{invrelabeling}
  Let $s > 0$ and let $\KEG_s$ be a size-$s$ {\kegname} generated by $\W$.
  Then, $\KEGD{\W}{s} = \EE[\REmbed{\CAM{\KEG_s},s}]$.
\end{lem}
\begin{proof}
  It suffices to show that $\RLabel[s]{\CAM{\KEG_s}} \equaldist \KEG_s$.

  Suppose $\KEG_s$ is generated as in \cref{KEGgen}.
  For simplicity of exposition, suppose that the generating graphex is $(0,0,W)$,
  and the associated latent Poisson process is $\PP_s$. 
  Let $\{\theta'_i\}_{i\in\Nats} \distiid \uniDist[0,s]$, and let
  $\PP'_s = \{(\theta'_i,\vartheta_i) \st (\theta_i,\vartheta_i) \in \PP_s \}$.
  By a property of the Poisson process, $\PP'_s \equaldist \PP_s$.
  Let $\KEG'_s$ be a size-$s$ {\kegname} generated using the same latent variables as $\KEG_s$,
  but with $\PP'_s$ replacing $\PP_s$.
  Then, by construction, $\KEG'_s \equaldist \RLabel[s]{\CAM{\KEG_s}}$.
  Moreover, $\KEG'_s$ is distributed as a size-$s$ {\kegname}, so $\KEG'_s \equaldist \KEG_s$.

  An essentially identical argument proves the result for a {\kegname} generated by the full graphex.
\end{proof}

The main sampling result is:

\begin{thm}\label{invsampling}
  Let $\W$ be a graphex,
  let $s > 0$ and $r \in [0,s]$, 
  let $G_s \dist \GPD{\W}{s}$,
  and let $G_r$ be an $r/s$-sampling of $G_s$.
  Then, $G_r \dist \GPD{\W}{s}$.
\end{thm}

\begin{proof}
  Let $\xi_s = \RLabel[s]{G_s}$. It is an obvious consequence of \cref{invrelabeling} that
  $\xi_s$ is equal in distribution to a size-$s$ {\kegname} generated by $\W$. Let $\xi_r$ be 
  the restriction of $\xi_s$ to $[0,r]^2$, so $\CAM{\xi_r} \dist \GPD{\W}{r}$.
  Each vertex of $\xi_s$ has a label in $[0,r]$ independently with probability $r/s$;
  thus, $\CAM{\xi_r} \equaldist G_r$.
\end{proof}

\section{Estimation with known sizes}\label{est_w_sizes}
This section explains our estimation results for the case where
the observations are $(G_1, s_1), (G_2, s_2), \dots$, where
$G_k = \CAM{\KEG_{s_k}}$ for some {\kegname} $\KEG$ generated according to a graphex $\W$
and some sequence $s_k \upto \infty$ in $\NNReals$.
We consider both the case of an arbitrary non-random divergent sequence and
the case where the sizes are taken to be the jumps of the {\kegname} (that is, the sizes
at which new edges enter the graph), in which case we denote the sequence as $\tau_1, \tau_2, \dots$
As motivated in the introduction, our notion of estimation is formalized as:
\begin{defn}
  Let $\W_1,\W_2,\dotsc$ be a sequence of graphexes.
  Write $\W_n \convEstGP \W$ as $n \to \infty$ when, for all $s\in\NNReals$, it holds that
  $\GPD{\W_n}{s} \to \GPD{\W}{s}$ weakly as $n \to \infty$.
\end{defn}
The goal of estimation is: given a sequence of observations $(G_1, s_1), (G_2, s_2),\dots$, 
produce 
\[
\hat{W}_{(G_k,s_k)}: \NNReals^2 \to [0,1]
\] 
such that $\hat{W}_{(G_k,s_k)} \convEstGP \W$ as $k \to \infty$, where the convergence
may be almost sure or merely in probability.

The main result of this section is that the dilated empirical graphons $\hat{W}_{(G_k,s_k)} \convEstGP \W$
for $(G_1, s_1), (G_2, s_2), \dots$ generated by a graphex $\W$; i.e. the dilated empirical graphon is a consistent estimator for $\W$.

We now turn to an intuitive description of the broad structure of the argument.
Conditional on $G_k$, let $\xi^k = \RLabel[s_k]{G_k}$ 
and let
$\REmbed{G_k, s_k}$ be the distribution of $\xi^{k}$ conditional on $G_k$.
The first convergence result, \cref{conv_of_emp_proc}, is that, almost surely,
the random distributions $\REmbed{G_k, s_k}$ converge weakly to $\law(\KEG) = \KEGDinf{\W}$.   
That is, for almost every realization of a {\kegname}, the point processes defined by
randomly labeling the observed finite graphs converge in distribution to the original {\kegname}. 
The analogous statement in the \iid sequence setting is that, given some
$(X_1, X_2, \dots)$ where $X_k \distiid P$,
and $\sigma_n$ a random permutation on $[1,\dots,n]$, 
the random distributions $\Pr(X_{\sigma_n(1)}, \dots, X_{\sigma_n(n)} \in \cdot \given X_1, \dots, X_n)$
converge weakly almost surely to $\Pr((X_1, X_2, \dots) \in \cdot)$ as $n \to \infty$.

The convergence in distribution of the point processes on $\NNReals^2$ is 
equivalent to convergence in distribution of the point processes restricted to $[0,r]^2$ for
every finite $r \in \NNReals$.
This perspective lends itself naturally to the interpretation of the limit result as a qualitative approximation theorem:
intuitively, $\Pr(\xi^k([0,r]^2 \cap \cdot) \in \cdot \given G_k)$ approximates $\KEGD{\W}{r}$, with the approximation becoming exact in the limit $r/s_k \to 0$.
This perspective also makes clear the first critical connection between estimation and sampling:
conditional on $G_k$, $\CAM{\xi^{k}([0,r]^2 \cap \cdot)}$ has the same distribution as an $r/s_k$-sampling of $G_k$. 

The second key observation is that, conditional on $G_k$,
a sample from $\GPD{\hat{W}_{(G_k, s_k)}}{r}$ 
may be generated by sampling $\poiDist(r/s_k \vertices(G_k))$ vertices with replacement from $G_k$
and returning the induced edge set.
The second step in the proof is to show that this sampling scheme is asymptotically
equivalent to $r/s_k$-sampling in the limit of $s_k \upto \infty$; this is the role
of \cref{w_wo_samp_equiv,poi_bin_equiv}.

\cref{estimation_w_size} then puts together these results to 
conclude that, almost surely, $\KEGDinf{\hat{W}_{(G_k, s_k)}} \to \KEGDinf{\W}$ weakly as $k \to \infty$.
Some additional technical rigmarole is required to show that this also gives convergence of the (unlabeled) random graphs.
This later convergence is the main result of this section, and is established in \cref{estimation_w_size_no_labels}.

\subsection{Convergence in Distribution of Random Embeddings}\label{conv_rembed}

This subsection uses results from the theory of 
distributional convergence of point processes to show that, almost surely,
$\REmbed{G_k, s_k} \to \GPD{\W}{\infty}$ weakly as $k \to \infty$. 

We will need the following definition and technical lemma:
A separating class for a locally compact second countable Hausdorff space $S$ is a class $\mathcal{U} \subset S$ such that
for any compact open sets with $K \subset G$ there is some $U \in \mathcal{U}$ with $K \subset U \subset G$.
\begin{lem}\label{sep_class_suffices}
Let $\phi, \phi_1, \phi_2,\dots$ be simple point processes on a locally compact second countable Hausdorff space $S$.
If 
\[
\phi_n(U) \convD \phi(U),\ n\to\infty
\]
weakly for all $U$ in some separating class for $S$ then
\[
\phi_n \convDist \phi,\ n\to\infty
\]
weakly.
\end{lem}
\begin{proof}
By \citep[][Thms.~16.28 and 16.29]{Kallenberg:2001}, 
it suffices to
check that $\Pr(\phi_n(U) = 0) \to \Pr(\phi(U) = 0)$ and that $\limsup_n \Pr(\phi_n(U) > 1) \le \Pr(\phi(U) > 1)$.
Because $\phi_n(U)$ is a non-negative integer a.s., both conditions are implied by $\phi_n(U) \convD \phi(U)$.
\end{proof}

\begin{thm}\label{conv_of_emp_proc}
Let $\KEG$ be a {\kegname} generated by a non-trivial graphex $\W$, let $s_1,s_2,\dots$ be some sequence in $\NNReals$
such that $s_k \upto \infty$ as $k\to\infty$ and let $G_k = \CAM{\KEG_{s_k}}$ for all $k$.
Then $\REmbed{G_k, s_k} \to \KEGDinf{\W}$ weakly almost surely.
\end{thm}

\begin{proof}
For each $k \in \Nats$, conditional on $G_k$, let $\xi^k$ be a point process with law $\REmbed{G_k, s_k}$.
Note that $\KEG \dist \KEGDinf{\W}$. Observe that the collection $\mathcal{U}$ of finite unions of rectangles with rational end points
is a separating class for $\NNReals^2$. Further, $\xi^k$ is simple for all $k \in \Nats$, as is $\KEG$. 
Thus by \cref{sep_class_suffices}, 
to show the claimed result it will suffice to show that, for all $U \in \mathcal{U}$,
$\Pr(\xi^k(U) \in \cdot \given G_k) \to \Pr(\KEG(U) \in \cdot)$ weakly as $k \to \infty$.

Fix $U$. To establish this condition we first show that
for all bounded continuous functions $f$, it holds that $\lim_{k\to\infty}\EE[f(\xi^k(U)) \given G_k] = \EE[f(\KEG(U))] \as$
Let $\mathcal{F}_{-s}$ be the partially labelled graph
derived from $\KEG$ by forgetting the labels of all nodes
with label $\theta_i<s$. 
Take $r \in \NNReals$ large enough so that $U \subset [0,r]^2$.
Then for $s_k > r$, 
\[
\EE[f(\KEG(U)) \given \mathcal{F}_{-{s_k}}] = \EE[f(\xi^k(U)) \given G_k].
\] 

Define $U_t = U + (t,t)$ for $t \in \NNReals$ and let 
\[
X^{(r)}_s = \frac {1} {s-r} \int_0^{s-r} f(\KEG(U_t)) \intd t.
\]
Observe that for $s$ such that $t \le s-r$, the joint exchangeability of $\KEG$ implies
\[
\EE[f(\KEG(U_t)) \given \mathcal{F}_{-s}] = \EE[f(\KEG(U)) \given \mathcal{F}_{-s}].
\] 
Moreover, by the linearity of conditional expectation,
for $s>r$, it holds that $\EE[X^{(r)}_s \given \mathcal{F}_{-s}] = \EE[f(\KEG(U)) \given \mathcal{F}_{-s}]$.

A standard result \citep[][Ex.~5.6.2]{Durrett:2010} shows that
$\lim_{k\to\infty}\EE[X^{(r)}_{s_k} \given \mathcal{F}_{-{s_k}}] = \EE[X^{(r)}_\infty \given \mathcal{F}_{-\infty}] \as$ if $X^{(r)}_{s_k} \to X^{(r)}_\infty \as $
and there is some integrable random variable that dominates $X^{(r)}_{s_k}$ for all $k$;
the second condition holds because $f$ is bounded.
Notice that $Y_t = f(\KEG(B_t))$ is a stationary stochastic process. Moreover, it's easy to see from the {\kegname} construction
that $Y_t$ and $Y_{t'}$ are independent whenever $\abs{t-t'}>r$, so $(Y_t)$ is mixing. 
The ergodic theorem then gives $\lim_{k\to\infty} X^{(r)}_{s_k} = \EE[f(\KEG(U))] \as$
This means 
\[
\lim_{k\to\infty}\EE[f(\xi^k(U)) \given G_k] \to \EE[f(\KEG(U))] \as,
\]
as promised.

For $l \in \NNInts$, let $f_l(\cdot) = 1[\cdot \le l]$, 
let $A^{(U)}_l$, for each $U \in \mathcal{U}$, be the set on which 
\[
\lim_{k\to\infty}\EE[f_l(\xi^k(U)) \given G_k] = \EE[f_l(\KEG(U))]
\] 
and let $A_U = \bigcap_l A^{(U)}_l$. We have shown that $\Pr(A^{(U)}_l)=1$, and so $\Pr(A_U) = 1$ and on $A_U$ it holds that $\lim_{k\to\infty} \Pr(\xi^k(U) \in \cdot \given G_k) = \Pr(\KEG(U) \in \cdot)$ weakly. Let $A = \bigcap_{U\in\mathcal{U}} A_U$, then $\Pr(A) = 1$ and on $A$ it holds that 
\[
\lim_{k\to\infty}\Pr(\xi^k(U) \in \cdot \given G_k) = \Pr(\KEG(U) \in \cdot)
\] 
weakly for all $U \in \mathcal{U}$, completing the proof.

\end{proof}

We need to do a little bit more work to show convergence in the case where the
observations are taken at the jumps of the {\kegname}.

\begin{thm}\label{conv_of_emp_proc_jump_times}
Let $\KEG$ be a {\kegname} generated by a non-trivial graphex $\W$,
and let $\tau_1, \tau_2, \dots$ be the jump times of $\KEG$. 
Let $G_k = \CAM{\KEG_{\tau_k}}$ for each $k \in \Nats$.
Then $\REmbed{G_k,\tau_k} \to \GPD{\W}{\infty}$ weakly almost surely as $k\to\infty$.
\end{thm}
\begin{proof}
For each $k\in\Nats$, let $\xi^k$ be a point process with law $\REmbed{G_k,\tau_k}$.

As in the proof of \cref{conv_of_emp_proc}, to establish the claim it suffices to show that, 
for all bounded continuous functions $f$ and all rectangles $U$, 
it holds that%
\[
\lim_{k\to\infty}\EE[f(\xi^k(U)) \given G_k, \tau_k] = \EE[f(\KEG(U))] \as
\]

Let $\mathcal{F}_{-s}$ be as in proof of \cref{conv_of_emp_proc}.
It is clear that $\mathcal{F}_{-\tau_k} \subset \mathcal{F}_{-\tau_(k-1)}$ for all $k$.
Because $U \subset [0,r]^2$ for some finite $r$ and $\tau_k \upto \infty \as$ as $k \to \infty$ it holds that
\[
\lim_{k\to\infty}\EE[f(\KEG(U)) \given \mathcal{F}_{-\tau_k}] = \lim_{k\to\infty} \EE[f(\xi^k(U)) \given G_k, \tau_k] \as
\]
Applying reverse martingale convergence to the l.h.s. we conclude the r.h.s. exists a.s.

It remains to identify the limit.  
To that end, we will define a coupling between the counts on test set $U$ 
at a subsequence of the jump times and the counts on $U$
at some deterministic sequence, which is known to converge
to the desired limit.
Let $s_k = \sum_{n=1}^k \frac{1}{n}$, 
let $\{\tau_{k_j}\}$ be a subsequence of the jump times defined such that
at most one point in $\{\tau_{k_j}\}$ lies in $[s_l,s_{l+1})$ for all $l$ and define
$s_{k_j}$ to be the subsequence of $\{s_k\}$ such that $s_{k_j}$ is the largest value in $\{s_k\}$ that is smaller than $\tau_{k_j}$. 
Intuitively, this gives a random subsequence of the jump times and a random subsequence of $\{s_k\}$ such that the
points $s_{k_j}$ and $\tau_{k_j}$ become arbitrarily close as $j\to\infty$.
For each $j \in \Nats$, let $G^s_j = \CAM{\KEG_{s_{k_j}}}$, and let $G^\tau_j = \CAM{\KEG_{\tau_{k_j}}}$.
 
By construction, $G^s_j \subset G^\tau_j$. Label the vertices of $G^\tau_j$ as
$1,\dots,\vertices(G^\tau_j)$ such that $1,\dots,\vertices(G^s_j)$ is the vertex set of $G^s_j$.
Let $\xi^{(s,j)} = \RLabel[s_{k_j}]{G^s_j}$, and let $\xi^{(\tau,j)} = \RLabel[\tau_{k_j}]{G^\tau_j}$. 
The occupancy counts of the test may then sampled according to:
\begin{enumerate}
\item $V_1,\dots,V_{\vertices(G^\tau_j)} \distiid \uniDist[0,1]$
\item $\xi^{(\tau,j)}(U) = |\{(v_i,v_j) \in e(G^\tau_j) \st (V_i \tau_{k_j}, V_j \tau_{k_j} ) \in U \}|$
\item $\xi^{(s,j)}(U) = |\{(v_i,v_j) \in e(G^s_j) \st (V_i s_{k_j}, V_j s_{k_j} ) \in U \}|$
\end{enumerate}

By construction, $G^\tau_j \exclude G^s_j$ is a star; call the center of this star $c$. Choosing $r$ such that $U \subset [0,r]^2$,
it is clear that if $V_c\tau_{k_j} \notin [0,r]$ then $\xi^{(\tau,j)}(U) \le \xi^{(s,j)}(U)$ under this coupling. 
The occupancy counts are the number of edges in random induced subgraphs given by including each vertex with probability 
$\frac{r}{\tau_{k_j}}$ and $\frac{r}{s_{k_j}}$ respectively. This perspective makes it clear that, conditional on $c$ not being included
when sampling from $G^\tau_j$, the counts will be equal as long as no vertices of the induced subgraph of $G^s_j$ are ``forgotten'' when 
the inclusion probability is reduced to $\frac{r}{\tau_{k_j}}$. 
The probability that $V_i s_{k_j} \in [0,r]$ but $V_i \tau_{k_j} \notin [0,r]$ 
is $\frac {r}{\tau_{k_j}s_{k_j}} (\tau_{k_j} - s_{k_j}) $. Moreover, there are at most $\xi^{(s,j)}(U)$ vertices in the subgraph sampled from
$G^s_j$ so, in particular,
\[
\EE[\xi^{(s,j)}(U) - \xi^{(\tau,j)}(U) \given E_{\bar{c}}, \KEG] \le \frac {r}{\tau_{k_j}s_{k_j}} (\tau_{k_j} - s_{k_j}) \xi^{(s,j)}(U),
\]
where $E_{\bar{c}}$ denotes the event that $c$ is not included in the subgraph sampled from $G^\tau_j$. 
Then, denoting the event $\{ \xi^{(s,j)}(U) = \xi^{(\tau,j)}(U) \}$ as $E_U$,

\[
\Pr(E_U \given \KEG) &\ge (1-\Pr(E_{\bar{c}}))(1-\Pr(\bar{E_U} \given \KEG, E_{\bar{c}}))\\
 &\ge (1-\frac {r}{\tau_{k_j}}) (1- \frac {r}{\tau_{k_j}s_{k_j}} (\tau_{k_j} - s_{k_j}) \xi^{(s,j)}(U).
\]

By construction, $\tau_{k_j} - s_{k_j} \le \frac{1}{k_j}$, so $\lim_{j\to\infty}\frac{\tau_{k_j} - s_{k_j}} {s_{k_j}} = 0$. In combination with 
$\lim_{j\to\infty}\xi^{(s,j)}(U) = \KEG(U) \as$ and the fact that $\KEG(U)$ is almost surely finite, the inequality we have just derived then implies that
\[
\lim_{j\to\infty}\Pr(\xi^{(s,j)}(U) \neq \xi^{(\tau,j)}(U) \given \KEG) = 0 \as
\]

In view of \cref{conv_of_emp_proc}, we thus have that
\[
\lim_{k\to\infty} \EE[f(\xi^k(U)) \given G_k, \tau_k] = \EE[f(\KEG(U))] \as,
\] 
as required.

\end{proof}

\subsection{Asymptotic Equivalence of Sampling Schemes}\label{samp_equiv}

As alluded to above, a key insight for showing that $\hat{W}_{(G_k,s_k)}$ is a valid estimator is that, 
conditional on $G_k$, a graph generated according to $\GPD{\hat{W}_{(G_k,s_k)}}{r}$
may be viewed as a random subgraph of $G_k$ induced by sampling 
$\poiDist(\frac{r}{s_k} \vertices(G_k))$ vertices from $G_k$ 
with replacement and returning the edge set of the vertex-induced subgraph.
The correctness of this scheme can be seen as follows:
\begin{enumerate}
\item Let $\PP$ be the latent Poisson process used to generate a sample from $\GPD{\hat{W}_{(G_k,s_k)}}{r}$, as in \cref{thm:graphex_rep_theorem},
  and let $\PP_r = \PP(\cdot \cap [0,r]^2)$.
  Because $\hat{W}_{(G_k,s_k)}$ has compact support $[0,\vertices(G_k)/s_k]^2$,
  only $\PP_r$ restricted to $[0,r] \times [0,\vertices(G_k)/s_k]$ can participate in the graph.
\item $\PP_r$ restricted to $[0,r] \times [0,\vertices(G_k)/s_k]$ may be generated by producing
  $J_{s_k,r} \dist \poiDist(r \vertices(G_k)/s_k)$ points $(\theta_i,\vartheta_j)$ where, conditional on $J_{s_k,r}$,
  $\theta_i \distiid \uniDist[0,r]$ and $\vartheta_i \distiid \uniDist[0,\vertices(G_k)/s_k]$, also independently of each other.
\item The $\{0,1\}$-valued structure of $\hat{W}_{(G_k,s_k)}$ means that choosing latent values $\vartheta_i \distiid \uniDist[0,\vertices(G_k)/s_k]$
  is equivalent to choosing vertices of $G_k$ uniformly at random with replacement.
\end{enumerate}

Our task is to show that the sampling scheme just described
is asymptotically equivalent to $r/s_k$-sampling of $G_k$. 
To that end, we observe that $r/s_k$-sampling is the same as sampling
$\binDist(\vertices(G_k),r/s_k)$ vertices of $G_k$ without replacement and returning
the induced edge set. This makes it clear that there are two main distinctions
between the sampling schemes: Binomial vs.\ Poisson number of vertices sampled,
and with vs.\ without replacement sampling. 
This motivates defining three distinct random subgraphs of $G_k$:  
\begin{enumerate}
\item $X^{(k)}_r$: Sample $\binDist(\vertices(G_k),\frac{r}{s_k})$ vertices without replacement and return the induced edge set
\item $H^{(k)}_r$: Sample $\binDist(\vertices(G_k),\frac{r}{s_k})$ vertices with replacement and return the induced edge set
\item $M^{(k)}_r$: Sample $\poiDist(\frac{r}{s_k}\vertices(G_k))$ vertices with replacement and return the induced edge set
\end{enumerate}
The observation that, conditional on $G_k$, $\xi^k_r \equaldist \RLabel[r]{X^{(k)}_r}$ makes the connection with the previous subsection clear.

Our aim is to show that when $r/s_k$ is small the different random subgraphs are all close in distribution.
A natural way to encode this is the total variation distance between their distributions.
However, because the distributions are themselves random ($G_k$ measurable) variables this is rather awkward.
It is instead convenient to work with couplings of the random subgraphs conditional on $G_k$;
this gives a natural notion of conditional total variation distance. See \citep{denHollander:2012} for an introduction to coupling arguments. 

Although we only need the sampling equivalence for sequences of graphs corresponding to a {\kegname},
we state the theorems for generic random graphs where possible.

The following result, which plays a similar role in the estimation
theory of graphons in the dense setting, is simply the asymptotic
equivalence of sampling with and without replacement.

\begin{lem}\label{w_wo_samp_equiv}
  Let $G$ be an almost surely finite random graph, with $\edges$ edges and $\vertices$ vertices.
  let $X_r$ be a random subgraph of $G$ given by sampling $\binDist(\vertices(G),\frac{r}{s})$ vertices without replacement and returning the induced edge set, and
  let $H_r$ be a random subgraph of $G$ given by sampling $\binDist(\vertices(G),\frac{r}{s})$ vertices with replacement and returning the induced edge set.
  Then there is a coupling such that 
  \[
  \Pr(H_r \neq X_r \given G) \le  2 \edges \Bigl (\frac{r^3}{s^3} + 2\frac{r^3}{s^3 \vertices^2} + 3 \frac{r^2}{s^2 \vertices} + \frac{r}{s \vertices^2} \Bigr)
  \] 

  Moreover, specializing to the {\kegname} case, with $H^{(k)}_r$ and $X^{(k)}_r$ defined as above, 
  under the same coupling, 
  \[
  \Pr(H^{(k)}_r \neq X^{(k)}_r \given G_k) \convPr 0,
  \]
  as $k \to \infty$.
  Further, if $\tau_1,\tau_2,\dots$ are the jump times of $\KEG$ then
  taking $s_k = \tau_k$ for all $k \in \Nats$, it holds that
  under this coupling
  \[
  \Pr(H^{(k)}_r \neq X^{(k)}_r \given G_k, \tau_k) \convPr 0,
  \]
  as $k \to \infty$.
\end{lem}
\begin{proof}
  Given $G$, we may sample $X_r$ according to the following scheme:
  \begin{enumerate}
    \item Sample $K_{s,r} \dist \binDist(\vertices, \frac {r}{s})$ 
    \item Sample a list $L = (L_1,L_2,\dots,L_{K_{s,r}})$ of vertices from $G$ without replacement 
    \item Return the edge set of the induced subgraph given by restricting $G$ to $L$
  \end{enumerate}
  
  Given $G$, we may sample $H_r$ similarly, except we use a list sampled
  with replacement; we couple $H_r$ and $X_r$ by coupling with and without
  replacement sampling of the vertex list.
  The following sampling scheme for a list
  $\tilde{L}$ returns a list that, given $G$, has the distribution
  of a length $K_{s,r}$ list of vertices sampled with replacement from $G$.
  Given $G$ we sample $\tilde{L}$ according to:
  \begin{enumerate}
    \item Sample $L$ as above
    \item $\tilde{L}_1 = L_1$
    \item For $j=1\dots K_{s,r}$, set $\tilde{L}_j = L_j$ with probability $1-\frac{j-1}{\vertices}$. Otherwise, 
      sample $\tilde{L}_j$ uniformly at random from $\{L_1,\dots,L_{j-1}\}$.
  \end{enumerate}
  $H_r$ is then sampled by returning the edge set of the induced subgraph given by taking
  $\tilde{L}$ as the vertex set. 

  Evidently, under this coupling,  $X_r = H_r$ as long as
  \begin{enumerate}
    \item Every entry of $L$ where $L \neq \tilde{L}$ does not participate in an edge in $X_r$
    \item Every entry of $\tilde{L}$ where $L \neq \tilde{L}$ does not participate in an edge in $X_r$
  \end{enumerate}
  Call the number of entries violating the first condition $F_1$ and the number of entries violating the second
  condition $F_2$, and let $N$ be the total number of entries where $L,\tilde{L}$ differ. Observe that when $K_{s,r}>0$, almost surely,
  \[
  \EE[F_1 \given v(H_r), N, K_{s,r}, G] &= \frac {v(X_r)} {K_{s,r}} N  \\ 
  \EE[F_2 \given v(H_r), N, K_{s,r}, G] &= \frac {v(X_r)} {K_{s,r}} N. 
  \]

  Further observe that because the sites where the lists disagree are chosen without reference to the graph structure it holds that $v(X_r)$ and $N$ are independent given $G$ and $K_{s,r}$, so 
  \[
    \EE[F_1 + F_2 \given v(X_r), K_{s,r}, G] = 2 \frac{v(X_r)}{K_{s,r}} \EE[N \given K_{s,r}, G].
  \]
  Moreover, almost surely,
  \[
    \EE[N \given K_{s,r}, G] &= \sum_{j=2}^{K_{s,r}}\frac{j-1}{\vertices} \\
    &= \frac{1}{2\vertices}(K_{s,r}^2-K_{s,r}). \label{w_wo_markov} 
  \]
  
Using Markov's inequality 
along with the observation that
  \[
  \Pr(X_r \neq H_r \given K_{s,r} < 2) = 0,
  \]
  and $K_{s,r}^2-K_{s,r} \le K_{s,r}^2$ on $K_{s,r} \ge 2$,
  \cref{w_wo_markov} implies that, almost surely,
  \[
  \Pr(X_r \neq H_r \given v(X_r), K_{s,r}, G) &\le \EE[F_1 + F_2 \given v(X_r), K_{s,r}, G] \\  
                          & \le \frac{K_{s,r}}{\vertices} v(X_r). \label{w_wo_smart_bound}
  \]

To prove the first assertion of the theorem statement,  
we now observe that $v(X_r) \le 2e(X_r)$ and $\EE[e(X_r) \given s, K_{s,r}, G] \le \edges \frac{K_{s,r}^2}{\vertices^2}$ (since each edge is included
with marginal probability at most $\frac{K_{s,r}^2}{\vertices^2}$), so it holds almost surely that 

\[
\Pr(X_r \neq H_r \given s, G) &\le \edges \frac{2}{\vertices^3} \EE[K_{s,r}^3 \given G]\\
                                      &= 2 \edges (\frac{r^3}{s^3} - 3 \frac{r^3}{s^3\vertices} + 2\frac{r^3}{s^3\vertices^2} + 3 \frac{r^2}{s^2\vertices} - 3 \frac{r^2}{s^2\vertices^2}
                                      + \frac{r}{s\vertices^2}).\label{w_wo_fin_bd}
\]

To prove the second assertion of the theorem statement we apply \cref{w_wo_smart_bound}
to the graph $G_k$ sampled at rate $r/s_k$, so
\[
\Pr(X^{(k)}_r \neq H^{(k)}_r \given v(X^{(k)}_r), K_{s_k,r}, G_k) \le \frac{K_{s_k,r}}{\vertices} v(X^{(k)}_r).
\]
Markov's inequality with $\EE[\frac{K_{s_k,r}}{\vertices(G_k)} \given G_k] = r/{s_k}$ implies that, given $G_k$, $\frac{K_{{s_k},r}}{\vertices(G_k)} \convPr 0$ as $k \to \infty$. Further,
by \cref{conv_of_emp_proc} and the observation that $X^{(k)}_r \equaldist \CAM{\xi^k(\cdot \cap [0,r]^2)}$ where $\xi^k \dist \REmbed{G_k,s_k}$,
it holds that $v(X^{(k)}_r) \convDist \vertices(\KEG_r) \as$ as $k \to \infty$. Since the integrability conditions on graphexes
guarantee that $\vertices(\KEG_r)$ is almost surely finite, 
\NA{we have
\[
\frac{K_{s_k,r}}{\vertices(G_k)} v(X^{(k)}_r) \convPr 0,
\]
as $k \to \infty$} and this implies,
\[
\Pr(X^{(k)}_r \neq H^{(k)}_r \given v(X^{(k)}_r), K_{s_k,r}, G_k) \convPr 0,
\] 
as $k \to \infty$.
Now,
\[
\Pr(X^{(k)}_r \neq H^{(k)}_r \given G_k) = \EE[\Pr(X^{(k)}_r \neq H^{(k)}_r \given v(X^{(k)}_r), K_{s_k,r}, G_k) \given G_k],
\]
and $\Pr(X^{(k)}_r \neq H^{(k)}_r \given G_k)$ is bounded by $1$ for all $k$, so the second claim follows by the dominated convergence theorem for conditional expectations, \citep[][Thm.~5.9]{Durrett:2010}.

The proof of the final claim goes through mutatis mutandis as the proof of the second assertion, 
subject to the observations that $\tau_k \upto \infty \as$, that we must condition on $\tau_k$ for each $k$, and that \cref{conv_of_emp_proc_jump_times} should be used in place of \cref{conv_of_emp_proc}.
\end{proof}

\begin{remark}
In the case that $\W = (0,0,W)$ and $W$ is integrable, it holds that $\vertices(G_k) = \Omega(s_k) \as$ and $\edges(G_k) = \Theta(s_k^2) \as$ \citep[][Props.~2.18~and~5.2]{Borgs:Chayes:Cohn:Holden:2016},
in which case the rate from the first part of the above lemma is $O(r^3/s_k)$. 
Note that in this case, the convergence in probability may be replaced by convergence almost surely.
This lemma is in fact the only component of the proof where a weakening of almost sure convergence is necessary, so (as remarked below),
whenever almost sure convergence holds for the equivalence of with and without replacement sampling, almost sure convergence holds for the
main estimation result. 
\end{remark}

It remains to show that the $\poiDist(r/s_k \vertices(G_k))$ and $\binDist(\vertices(G_k),r/s_k)$ samplings are asymptotically equivalent.
Note that the rate ($\vertices(G_k)/s_k$) at which the empirical graphon is dilated guarantees that the expected
number of vertices sampled according to each scheme is equal; this is the reason that this rate was chosen.

\begin{lem}\label{poi_bin_equiv}
  Let $G$ be an almost surely finite random graph with $\vertices$ vertices.
  Let $H_r$ be a random subgraph of $G$ given by sampling $\binDist(\vertices, \frac{r}{s})$ vertices with replacement and returning the induced edge set,
  and let $M_r$ be a random subgraph of $G$ given by sampling $\poiDist(\vertices \frac{r}{s})$ vertices with replacement and returning the induced edge set.
  Then there is a coupling such that
   \[
   \Pr(H_r \neq M_r \given G) \le \frac r s \as
   \] 

\end{lem}

\begin{proof}

  Conditional on $G$, $H_r$ may be sampled by:
  \begin{enumerate}
  \item sample $K_{s,r} \dist \binDist(\vertices,r/s)$ vertices with replacement from $G$;
  \item return the edge set of the induced subgraph.
  \end{enumerate}

  Conditional $G$, $M_r$ may be sampled by:
  \begin{enumerate}
  \item sample $J_{s,r}\dist\poiDist(r \frac{\vertices}{s})$ vertices with replacement from $G$.
  \item return the edge set of the induced subgraph.
  \end{enumerate}

  Comparing the two sampling schemes, it is immediate that there is a
  coupling such that
  \[
  \Pr(H_r \neq M_r \given G) &\le \Pr(K_{s,r} \neq J_{s,r} \given G).
  \]

  Note that $\expect{K_{s,r} \given G}=\expect{J_{s,r} \given G}$.
  The approximation of a sum of Bernoulli random variables by a Poisson
with the same expectation as the sum is well studied: 
if $X_1,\dots,X_l$
  are independent random variables with $\bernDist(p_i)$ distributions
  such that $\lambda=\sum_{i=1}^lp_i$ and $T\dist\poiDist(\lambda)$
  then there is a coupling~\citep[][Sec.~5.3]{denHollander:2012} such that 
$\Pr(T\neq\sum_{i=1}^lX_i) \le \frac 1 \lambda \sum_{i=1}^sp_i^2$. 
This implies that there is a coupling of
  $K_{s,r}$ and $J_{s,r}$ such that
  \[
  \Pr(K_{s,r} \neq J_{s,r} \given G) &\le \frac r s,
  \]
  completing the proof.
\end{proof}

\subsection{Estimating $\W$}

We now combine our results to show that the law of the {\kegname} generated by the empirical graphex 
converges to the law of a graphex process generated by the underlying $\W$.

There is an immediate subtlety to address: \cref{conv_rembed} deals with convergence in distribution of point processes
(i.e., labeled graphs), and \cref{samp_equiv} deals with convergence in distribution of unlabeled graphs.
We first give the main convergence result for the point process case.
In order to state this result compactly it is convenient to metrize weak convergence.
To this end, we recall that the space of boundedly finite measures may be equipped with a metric such that it is a
complete separable metric space \citep[][Eqn.~A.2.6]{Daley:Vere-Jones:2003:v1}.
Let $\dProk{\cdot}{\cdot}$ be the Prokhorov
metric on the space of probability measures over boundedly finite measures induced by the aforementioned metric. 
Then $\dProk{\cdot}{\cdot}$ metrizes weak convergence:
i.e., for a sequence of boundedly finite random measures $\{\Pi_n\}$ it holds that
$\Pi_n \convDist \Pi$ as $n \to \infty$ if and only if $\dProk{\law(\Pi_n)}{\law(\Pi)} \to 0$ as $n \to \infty$.

\begin{thm}\label{estimation_w_size}
Let $\KEG$ be a {\kegname} generated by non-trivial graphex $\W$ and let $s_1,s_2,\dots$ be a (possibly random) sequence in $\NNReals$ 
such that $s_k \upto \infty$ almost surely as $k \to \infty$.
Let $G_k = \CAM{\KEG_{s_k}}$ for $k \in \Nats$.
Suppose that either
\begin{enumerate}
\item $(s_k)$ is independent of $\KEG_k$, or
\item $s_k = \tau_k$ for all $k \in \Nats$, where $\tau_1, \tau_2, \dots$ are the jump times of $\KEG$. 
\end{enumerate}
Then
\[
\dProk{\KEGDinf{\hat{W}_{(G_k,s_k)}}} {\KEGDinf{\W}}  \convPr 0,
\]
as $k \to \infty$. 
\end{thm}
\begin{proof}

For notational simplicity, we treat the deterministic index case first. 

For $r \in \NNReals$, let $\REmbed{G_k, s_k}|_r$ denote the probability measure over point processes on $[0,r]^2$
induced by generating a point process according to $\REmbed{G_k, s_k}$ and restricting to $[0,r]^2$.

By the triangle inequality,
\[
\dProk{\KEGD{\hat{W}_{(G_k,s_k)}}{r}}{\KEGD{\W}{r}} &\le \dProk{\KEGD{\hat{W}_{(G_k,s_k)}}{r}}{\REmbed{G_k, s_k}|_r} \\ 
&\quad + \dProk{\REmbed{G_k, s_k}|_r}{\KEGD{\W}{r}}.
\] 

Conditional on $G_k$ and $s_k$, let $X^k_r$ be an $r/s_k$-sampling of $G_k$ and 
let $M^k_r$ be a random subgraph of $G_k$ given by sampling $\poiDist(\vertices(G_k)r/s_k)$ vertices
with replacement and returning the edge set of the vertex-induced subgraph.
By \cref{w_wo_samp_equiv,poi_bin_equiv} it holds that there is a sequence of couplings such that
\[
\Pr(M^{k}_r \neq X^{k}_r \given G_k, s_k) \convPr 0,\ k\to\infty.
\]
Observe that
$\KEG^k_r \equaldist \RLabel[r]{M^k_r, \{U_i\}}$ and
$\xi^{k}_r \equaldist \RLabel[r]{X^k_r, \{U_i\}}$, where $U_i \distiid \uniDist[0,r]$ for $i \in \Nats$.
Here $\xi^{k}_r$  
is a random labeling of $G_k$, as in \cref{conv_of_emp_proc}.
Thus, the couplings of the unlabeled graphs lift to couplings of the point processes such that
\[
\Pr(\KEG^k_r \neq \xi^k_r \given G_k, s_k) \convPr 0,\ k\to\infty.
\]
The relationship between couplings and total variation distance then implies 
\[
\|\KEGD{\hat{W}_{(G_k,s_k)}}{r} - \REmbed{G_k, s_k}|_r\|_{\mathrm{TV}} \convPr 0,\ k\to\infty,
\]
so also,
\[
\dProk{\KEGD{\hat{W}_{(G_k,s_k)}}{r}}{\REmbed{G_k, s_k}|_r} \convPr 0,\ k\to\infty. \label{emp_graphex_emp_dist_conv}
\]

Second, by \cref{conv_of_emp_proc},
\[
\dProk{\REmbed{G_k, s_k}|_r}{\KEGDinf{\W}} \convPr 0,\ k\to\infty.  \label{emp_dist_conv}
\]

Thus,
\[
\dProk{\KEGD{\hat{W}_{(G_k,s_k)}}{r}}{\KEGD{\W}{r}} \convPr 0,\ k\to\infty.
\]

By \citep[][Lem.~4.4]{Kallenberg:2001}, convergence in probability for each element of a sequence lifts to convergence in probability of the entire sequence:
\[
(\dProk{\KEGD{\hat{W}_{(G_k,s_k)}}{1}}{\KEGD{\W}{1}}, \dProk{\KEGD{\hat{W}_{(G_k,s_k)}}{2}}{\KEGD{\W}{2}}, \cdots ) \convPr 0,\ k\to\infty.
\]
As the space of boundedly finite measures on $\NNReals^2$ is homeomorphic 
to the space of sequences of restrictions of boundedly finite measures to $[0,r]^2$, for $r \in \Nats$,
it follows that 
\[
\dProk{\KEGDinf{\hat{W}_{(G_k,s_k)}}}{\KEGDinf{\W}} \convPr 0,\ k\to\infty.
\]

The same proof mutatis mutandis applies for convergence along the jump times. The main substitution is the use of \cref{conv_of_emp_proc_jump_times} in place of \cref{conv_of_emp_proc}.
\end{proof}

\begin{remark}
For graphexes such that $\edges(G_k)/s_k^3 \to 0 \as $ and $\vertices(G_k) = \Omega(s_k)$ the convergence in probability above can be replaced by almost sure convergence by replacing all the
convergence in probability statements in the body of the proof by almost sure statements. This class of such graphexes includes all integrable $(0,0,W)$.
\end{remark}

\newcommand{\GPDist}{Q}

We now turn to the analogous result for the case of unlabeled graphs generated by the dilated empirical graphon.
We begin with a technical lemma that allows us to deduce convergence in distribution of unlabeled graphs from
convergence in distribution of the associated adjacency measures.
Note that the map taking an adjacency measure to 
its associated graph is measurable, but not continuous, and so this result does not follow from a naive application of the continuous mapping theorem.

\begin{lem}\label{backwards_cont_mapping}
Let $S$ be a discrete space, $T$ a metric space, $Q_1,Q_2,\dots$ a tight sequence of probability measures on $S$,
and $K$ a probability kernel from $S$ to $T$,
such that $K$ is injective when considered as a map from probability measures on $S$ to probability measures on $T$.
If $Q_1K, Q_2K, \dots$ converge weakly to $QK$ then $Q_1,Q_2,\dots$ converges weakly to $Q$.
\end{lem}
\begin{proof}
Assume otherwise.
Case 1: $Q_n \to Q' \neq Q$ weakly. 
By \citep[][Lem.~16.24]{Kallenberg:2001} and the discreteness of $S$,
$Q_nK \to Q'K$ weakly. Since $K$ is injective $Q'K \neq QK$, a contradiction.

Case 2: $Q_n$ does not converge weakly. Since the sequence $Q_n$ is tight it does converge subsequentially. Choose two
infinite subsequences $Q_{i_1},Q_{i_2},\dots$ and $Q_{j_1},Q_{j_2},\dots$ with respective limits $Q',Q''$ with $Q' \neq Q''$. 
But then,
by \citep[][Lem.~16.24]{Kallenberg:2001} and the discreteness of $S$, $Q'_{i_k}K \to Q'K$ and $Q''_{j_k}K \to Q''K$, hence $Q' K = QK = Q''K$, but $K$ is injective, hence $Q'=Q=Q''$, a contradiction.
\end{proof}

The motivating application of this last lemma is showing that
a sequence of graphs $G_1, G_2, \dots$ converge in distribution
if and only if their random labelings into $[0,s]$ for some $s$
also converge in distribution.
To parse the following theorem, note that when $G$ is a finite random graph,
and $s \in \NNReals$, 
then $\Pr(G \in \cdot)\REmbed{\cdot,s} = \Pr(\RLabel[s]{G} \in \cdot)$.  
\begin{lem}\label{graph_conv_equiv_adj_meas_conv}
Let $K_s(\cdot) = \REmbed{\cdot,s}$ for $s \in \NNReals$,
let $Q, Q_1, Q_2, \dots$ be probability measures on the space of almost surely finite random graphs,
let $\zeta_k = Q_k K_s$ and let $\zeta = Q K_s$.
Then, $Q_k \to Q$ weakly as $k \to \infty$ if and only if
$\zeta_k \to \zeta$ weakly as $k \to \infty$.
\end{lem}
\begin{proof}
The forward direction (convergence in distribution of the random graphs implies convergence in distribution of the random adjacency measures)
follows immediately from the discreteness of the space of finite graphs and \citep[][Lem.~16.24]{Kallenberg:2001}.

Conversely, suppose that $\zeta_k \to \zeta$ weakly as $k \to \infty$,
and, for every $n \in \Nats$, let $E_n$ be the set of adjacency measures $\xi$ such that 
$\xi ([0,s]^2) \le n$, i.e., $E_n$ is the event that the graph has fewer than $n$ edges.
Note that $E_n$ is a $\zeta$-continuity set by the definition of $K_s$, and therefore, by weak convergence,
$\zeta_k(E_n) \to \zeta(E_n)$ as $ k \to \infty$ for every $n \in \Nats$.
Let $E'_n$ be the set of graphs with fewer than $n$ edges. By definition,
$Q_k(E'_n) = \zeta_k(E_n)$ and $Q(E'_n) = \zeta(E_n)$, hence
$Q_k(E'_n) \to Q(E'_n)$.
But $E'_n$ is a finite (hence, compact) set, hence $\{Q_k\}_{k \in \Nats}$ is tight.
Noting in addition that $K_s$ is injective, the result follows from \cref{backwards_cont_mapping}.
\end{proof}

The following theorem is a formalization of $\hat{W}_{(G_k, s_k)} \convEstGP \W$ as $k \to \infty$ in probability:

\begin{thm}\label{estimation_w_size_no_labels}
Let $\KEG$ be a {\kegname} generated by non-trivial graphex $\W$ and let $s_1,s_2,\dots$ be a (possibly random) sequence in $\NNReals$ 
such that $s_k \upto \infty$ almost surely as $k \to \infty$.
Let $G_k = \CAM{\KEG_{s_k}}$ for $k \in \Nats$.
Suppose that either
\begin{enumerate}
\item $(s_k)$ is independent of $\KEG_k$, or
\item $s_k = \tau_k$ for all $k \in \Nats$, where $\tau_1, \tau_2, \dots$ are the jump times of $\KEG$. 
\end{enumerate}
Then, for every infinite sequence $N \subseteq \Nats$,
there exists an infinite subsequence $N' \subseteq N$,
such that 
\[
\hat{W}_{(G_k, s_k)} \convEstGP \W \as
\]
along $N'$.

\end{thm}

\begin{proof}

We first treat the case (1) where the times $(s_k)$ are independent of $\KEG$.

Let $N \subseteq \Nats$ be an infinite sequence.
\cref{estimation_w_size} 
implies that
there is some infinite subsequence $N' \subseteq N$ such that, for all $r\in\NNReals$, 
$\KEGD{\hat{W}_{(G_k,s_k)}}{r} \to \KEGD{\W}{r}$ weakly almost surely along $N'$.

Let $r \in \NNReals$ and $K_r(\cdot) = \REmbed{\cdot,r}$. For all $k \in N'$,
\[
\GPD{\hat{W}_{(G_k,s_k)}}{r} K_r = \KEGD{\hat{W}_{(G_k,s_k)}}{r} \as,
\]
and $\GPD{\W}{r} K_r = \KEGD{\W}{r}$.
Moreover, the graph corresponding to a size-$r$ {\kegname} is almost surely finite.
Thus \cref{graph_conv_equiv_adj_meas_conv} applies and we have that $\GPD{\hat{W}_{(G_k,s_k)}}{r} \to \GPD{\W}{r}$ weakly a.s. along $N'$. 
This holds for all $r\in\NNReals$, so we have even that $\hat{W}_{(G_k, s_k)} \convEstGP \W \as$ along $N'$.

The same proof mutatis mutandis applies for convergence along the jump times.
\end{proof}

\begin{remark}
For graphexes such that $\edges(G_k)/s_k^3 \to 0 \as $ and $\vertices(G_k) = \Omega(s_k)$,
\cref{estimation_w_size} implies that $\hat{W}_{(G_k, s_k)} \convEstGP \W$ as $k \to \infty$ almost surely and not merely in probability.
The class of graphexes with these two properties includes all graphexes of the form $(0,0,W)$ for integrable $W$.
\end{remark}

\section{Estimation for unknown sizes}\label{est_wo_sizes}

We now turn to the case where only the graph structure of the {\kegname} is observed, 
rather than the graph structure and the sizes of the observation.

We first show how distinct adjacency measures can give rise to the same graph sequence. 
For a measurable map $\phi : \NNReals \to \NNReals$ and adjacency measure $\xi$, 
define $\xi^\phi$ to be the measure given by 
$\xi^{\phi}(A \times B) = \xi (\phi^{-1}(A) \times \phi^{-1}(B))$, for every measurable $A, B \subseteq \NNReals$.
The graph sequences underlying an adjacency measure $\xi$ is invariant to the action $\phi \mapsto \xi^\phi$ of every strictly monotonic and increasing function $\phi$.

\begin{prop}\label{invgs}
Let $\xi$ be an adjacency measure and
let $\phi: \NNReals \to \NNReals$ be strictly monotonic and increasing.
Then $\GS{\xi} = \GS{\xi^\phi}$.
\end{prop}
\begin{proof}
Let $\{\tau_k\}$ and $\{\tau^\phi_k\}$ be the stopping sizes of $\xi$ and $\xi^\phi$, respectively.

Since $\phi$ is strictly monotonic it is also invertible. 
From this observation it is easily seen that $(\theta_i,\theta_j)$ is an atom of $\xi$
if and only if $(\phi(\theta_i),\phi(\theta_j))$ is an atom of $\xi^\phi$.
It is then clear that, for all $k \in \Nats$, $\phi(\tau_k) = \tau^\phi_k$  
and, moreover, the graph structure of $\{(x_i,\tau_k) : (x_i,\tau_k) \in \xi \}$ 
is equal to the graph structure of ${(y_i,\tau^\phi_k) : (y_i,\tau^\phi_k) \in \xi^\phi}$.
That is, the subgraph of all edges added at the $k$th step is equal for both graph sequences,
for all $k \in \Nats$. Moreover, the first entry of each graph sequence is 
(obviously) equal to the subgraph of all edges added at the first step.
The proof is then completed by induction.
\end{proof}

If $\phi$ is an arbitrary strictly monotonic mapping and $\xi$ is an exchangeable adjacency measure,
it will not generally be the case that $\xi^\phi$ is exchangeable. 
One family of mappings that preserves exchangeability is $\phi(x) = cx$, for $c \in \NNReals$.
We define the $c$-dilation of an adjacency measure $\xi$ to be the adjacency measure $\xi^\phi$ for this map. 
Because $\xi^\phi$ is exchangeable there is some graphex $\W'$ that generates it:
the next result shows that the $\frac 1 c$-dilation of a {\kegname} corresponds to a $c$-dilation of its graphex.

\newcommand{\IPP}{\PP^i}
\begin{lem}\label{cdilation}
Let $\KEG$ be a {\kegname} with graphex $\W = (I, S, W)$.
Then the $\frac 1 c$-dilation of $\KEG$ is a {\kegname} $\KEG'$ with generating graphex $\W' = (I',S',W')$ where $I' = c^{2} I$, $S'(x)=cS(x/c)$, and $W'(x,y)=W(x/c,y/c)$.
\end{lem}
\begin{proof}
Let $\KEG$ be a {\kegname} generated by $W$ with latent Poisson processes $\PP, \PP_1, \PP_2, \dotsc$, and  $\IPP$ on $\NNReals^2$ and $\NNReals^3$, respectively. 
Define $f(\PP) = \{(\frac 1 c \theta, c \vartheta) \st (\theta,\vartheta) \in \PP\}$, 
define $f(\PP_n) = \{(\frac 1 c \sigma, c \chi) \st (\sigma,\chi) \in \PP_n\}$, for $n \in \Nats$, 
and define $f(\IPP) = \{(\frac 1 c \rho, \frac 1 c \rho', c^2 \eta) \st (\rho,\rho',\eta) \in \IPP \}$. 
Note that $f(\PP)$ and $f(\PP_n)$, for $n \in \Nats$, are unit-rate Poisson processes on $\NNReals^2$,
and $f(\IPP)$ is a unit-rate Poisson process on $\NNReals^3$.
Indeed, the joint law of $(\PP,\PP_1,\PP_2,\dotsc,\IPP)$ is the same as that of 
$(f(\PP),f(\PP_1),f(\PP_2),\dotsc,f(\IPP))$.

Then $\KEG'$, the $\frac 1 c $-dilation of $\KEG$, is the {\kegname} generated by 
$\W'$ with latent Poisson processes
$f(\PP),f(\PP_1),f(\PP_2),\dotsc,f(\IPP)$
\emph{reusing} the same i.i.d.\ collection $(\zeta_{\{i,j\}})$ in $[0,1]$ as was used to generate $\KEG$.
To see this, note that
$\KEG'$ includes edge $(\frac 1 c \theta_i,\frac 1 c \theta_j)$
 if and only if 
$\zeta_{\{i,j\}} \le W'(c\vartheta_i,c \vartheta_j)$ 
 if and only if 
$\zeta_{\{i,j\}} \le W(\vartheta_i,\vartheta_j)$ 
 if and only if 
$\KEG$ includes edge $(\theta_i,\theta_j)$. 
Similarly, 
$\KEG'$ includes edge $(\frac 1 c \theta_i, \frac 1 c \sigma_{ij})$ 
 if and only if 
$ c \chi_{ij} \le S'(c\vartheta) = c S(\vartheta)$
 if and only if 
$ \chi_{ij} \le S(\vartheta)$
 if and only if
$\KEG$ includes edge $(\theta_i, \sigma_{ij})$.
Finally,
$\KEG'$ includes edge $(\frac 1 c \rho,\frac 1 c \rho')$
 if and only if 
$c^2\eta \le I' = c^2 I$
 if and only if 
$\KEG$ includes edge $(\rho,\rho')$.
Thus $\KEG'$ is a $\frac 1 c$-dilation of $\KEG$, as was to be shown.
\end{proof}

Define the \defnphrase{$c$-dilation of a graphex $\W$} to be the graphex $\W'$ defined in the statement of \cref{cdilation}.
We have the following consequence: 

\begin{thm}\label{invdilation}
Let $\W$ be a graphex, 
let $\W'$ be the $c$-dilation of $\W$ for some $c>0$,
and let $\KEG$ and $\KEG'$ be {\kegnames} with graphexes $\W$ and $\W'$, respectively.
Then $\GS{\KEG} \equaldist \GS{\KEG'}$.
\end{thm}
\begin{proof}
Follows immediately from \cref{cdilation,invgs}.
\end{proof}

As a consequence of this result, if the observed data is the graph sequence---that is, if the size $s$ is unknown---then the dilation of the
generating graphex is not identifiable. Therefore, the notion of estimation that we used in the known-size setting is not appropriate, because it requires
$\GP_r(\W_n) \convDist \GP_r(\W)$ as $n \to \infty$ for all sizes $r\in\NNReals$.

The appropriate notion of estimation in this setting is then: 
\begin{defn}
  Let $\W, \W_1,\W_2,\dotsc$ be a sequence of graphexes,
  and let $\KEG,\KEG^1, \KEG^2,\dotsc$ be {\kegnames} generated by each graphex.
  Write $\W_k \convEstGS \W$ as $k \to \infty$ when
  $\GS{\KEG^k} \convDist \GS{\KEG}$ as $k \to \infty$.
\end{defn}
Note that this is equivalent to requiring convergence in distribution of the length-$l$ prefixes of the graph sequences, for all $l \in \Nats$.
Intuitively, a length-$l$ graph sequence generated by the estimator is close in distribution to a length-$l$ graph sequence generated
by the true graphex, provided the observed graph is large enough.
This perspective explains how a sequence of compactly supported graphexes can estimate a graphex that is not itself compactly supported.
The following is immediate from \cref{invdilation}.
\begin{cor}\label{convdil}
Let $\W, \W_1, \W_2, \dots$ be a sequence of graphexes, let $c,c_1,c_2,\dots >0$, and let $W^c,W_1^{c_1},W_2^{c_2},\dots$ be the corresponding dilations.
Then $\W_k \convEstGS \W$ as $k \to \infty$ if and only if  $\W_k^{c_k} \convEstGS \W^c$ as $k \to \infty$.
\end{cor}

Intuitively speaking, $\W_k \convEstGS \W$ as $k \to \infty$
demands less than $\W_k \convEstGP \W$ as $k \to \infty$,
because in the former case we don't need to find a correct rate of dilation for the graphex.
The intuition that convergence in distribution of the graph sequence
is weaker than convergence in distribution of $(\CAM{\KEG_s})_{s\in\NNReals}$ is
borne out by the next lemma: 
\begin{lem}\label{gp_conv_implies_gs_conv}
Let $\W, \W_1, \W_2, \dots$ be graphexes where $\W$ is non-trivial and $\W_k \convEstGP \W$ as $k \to \infty$.
Then $\W_k \convEstGS \W$ as $k \to \infty$.
\end{lem}
\begin{proof}
Let $\KEG^k$ be {\kegnames} generated by $\W_k$, and let $\KEG$ be generated by $\W$. 
For $n \in \Nats$, let $G^k_n = \CAM{\KEG^k_n}$, and let $G_n = \CAM{\KEG_n}$.

Consider the sequence %
$H^k_n = (\GS{\KEG^k_1}, \GS{\KEG^k_2}, \dots, \GS{\KEG^k_n})$,
where each entry is itself an a.s.~finite graph sequence and entry $j$ is a prefix of entry $j+1$.
Let $\eta^k_n = \Pr(H^k_n \in \cdot)$, and let $\eta_n = \Pr((\GS{\KEG_1}, \GS{\KEG_2}, \dots, \GS{\KEG_n}) \in \cdot)$.
Intuitively speaking, we are breaking up the graph sequence of the entire {\kegname} into the graph 
sequences up to size $1,2,\dots$
and $\eta_n$ is the joint distribution of the first $n$ of these partial graph sequences.
Our short term goal is to show that $\eta^k_n \to \eta_n$ weakly as $k \to \infty$.

To that end, let $G$ be a finite graph and consider the random variable
\[
L_n(G) = (\GS{\RLabel[n]{G}([0,j)^2 \cap \cdot)})_{j=1,\dots,n}.
\]
This is a nested sequence of graph sequences given by mapping $G$ to an adjacency measure on $[0,n)^2$ and then
returning the sequence of graph sequences corresponding to this adjacency matrix at sizes $1,\dots,n$.
The significance of this construction is that we may use it to define a probability kernel,
\[
K_n(G,\cdot) = \Pr(L_n(G) \in \cdot),
\]
such that that $\Pr(G^k_n \in \cdot) K_n = \EE K_n(G^k_n,\cdot) = \eta^k_n$ and $\Pr(G_n \in \cdot) K_n = \EE K_n(G_n,\cdot) = \eta_n$.
By assumption, we have $\W_k \convEstGP \W$ as $k \to \infty$, whence $G^k_n \convDist G_n$ as $k \to \infty$.
By the discreteness of the space of finite graphs and \citep[][Lem.~16.24]{Kallenberg:2001}
it then holds that,
\[
\Pr(G^k_n \in \cdot) K_n \to \Pr(G_n \in \cdot)K_n,
\]
weakly as $k\to\infty$.
It thus holds by the construction of $K_n$ that 
\[
\eta^k_n \to \eta_n, \label{latent_size_conv}
\] 
weakly as $k\to\infty$.

We now have that an arbitrary length prefix of the graph sequence converges in distribution, when
the notion of length is given by the latent sizes. 
It remains to argue that this convergence holds for arbitrary prefixes in the usual sequence sense.
To that end, we observe that because \cref{latent_size_conv} holds for all $n \in \Nats$, by \citep[][Thm.~4.29]{Kallenberg:2001} it further holds that
\[
(\GS{\KEG^k_1}, \GS{\KEG^k_2}, \dots) \convDist (\GS{\KEG_1}, \GS{\KEG_2}, \dots),\ k\to\infty.\label{nested_gs_conv}
\] 
There is function $f$ such that, for every locally finite but infinite adjacency measure $\xi$ on $\NNReals^2$, with restrictions $\xi_j$ to $[0,j)^2$,
\[
   f(\GS{\xi_1}, \GS{\xi_2}, \dots) = \GS{\xi}
\]
and $f$ is even continuous because every finite prefix of $\GS{\xi}$ is determined by some finite prefix of the left hand side.
Extend $f$ to the space of all nested graph sequences arbitrarily.
Note that $f$ is continuous at $\KEG$ a.s.\ because $\KEG$ is a.s.\ locally finite and $\|W\|_1 > 0$ implies $\KEG$ is a.s.\ infinite.
Hence, the result follows by the continuous mapping theorem \citep[][Thm.~4.27]{Kallenberg:2001}.
\end{proof}

We now turn to establishing the main estimation result for the setting where the sizes are not
included as part of the observation.
In this setting, 
the observations are increasing sequences of graphs $G_1, G_2, \dots$. There are two natural models for the observations:
In one model,
$G_k = \CAM{\KEG_{s_k}}$
for some {\kegname} $\KEG$ and (possibly random, independent, and a.s.) increasing and diverging sequence of sizes $s_1, s_2, \dots$. 
Alternatively, in the other model,
the sequence $G_1, G_2, \dotsc$ is the graph sequence $\GS{\KEG}$ of some {\kegname} $\KEG$.

A natural estimator is the empirical graphon, $\empGraphon_{G_k}$, reflecting the intuition that
the dilation necessary in the previous section for convergence of the generated {\kegname}
is irrelevant for convergence in distribution of the associated graph sequence.
Somewhat more precisely, we view the empirical graphon as the canonical representative of the equivalence class of graphons given by equating graphons that induce the same distribution on graph sequences.
The main result of this section is that $\empGraphon_{G_k} \convEstGS \W$ as $k \to \infty$ in probability,
for either of the natural models for the observed sequence $G_1, G_2, \dotsc$.

\begin{thm}\label{estimation_wo_size}
Let $\KEG$ be a {\kegname} generated by some non-trivial graphex $\W$ and
let $G_1, G_2, \dots$ be some sequence of graphs such that either
\begin{enumerate}
\item There is some random sequence $(s_k)$, independent from $\KEG$, such that
$s_k \upto \infty$ a.s.\ and $G_k = \CAM{\KEG_{s_k}}$ for all $k \in \Nats$, or
\item $(G_1, G_2, \dots) = \GS{\KEG}$.
\end{enumerate}
Then, for every infinite sequence $N \subseteq \Nats$, there is an infinite subsequence $N' \subseteq N$, such that
\[
\empGraphon_{G_k} \convEstGS \W \as,
\] 
along $N'$.
\end{thm}

\begin{proof}
We prove case (1). Case (2) follows mutatis mutandis,
substituting $\tau_k$ for $s_k$.

Let $\hat{W}_{(G_k,s_k)}$ denote the dilated empirical graphon of $G_k$ with observation size $s_k$.
By \cref{estimation_w_size_no_labels}, for every sequence $N \subseteq \Nats$, there is an infinite subsequence $N' \subseteq N$, 
such that $\hat{W}_{(G_k,s_k)} \convEstGP \W$ along $N'$.
By \cref{gp_conv_implies_gs_conv} and $\W$ being non-trivial, this implies that $\hat{W}_{(G_k,s_k)} \convEstGS \W$ along $N'$.
For every $k$, $\empGraphon_{G_k}$ is some dilation of $\hat{W}_{(G_k,s_k)}$, hence, the result follows by \cref{convdil}.
\end{proof}

\section*{Acknowledgements}
The authors would like to thank
Christian Borgs, Jennifer Chayes, and Henry Cohn
for helpful discussions.
This work was supported by U.S. Air Force Office of Scientific
Research grant \#FA9550-15-1-0074.

\printbibliography

\vfill

\end{document}